\pdfoutput=1
\documentclass[reqno]{amsart}
\usepackage{amsmath,amssymb,amsthm,verbatim, csquotes,mathrsfs,stmaryrd}
\usepackage{hyperref}
\usepackage{xcolor}
\usepackage{esint}
\usepackage{mathtools}
\usepackage{graphicx}
\usepackage{float}
\usepackage{caption}
\mathtoolsset{showonlyrefs}
\usepackage[left=2.5cm, right=2.5cm, top=3.5cm, bottom=4cm]{geometry}

\newtheorem{theorem}{Theorem}[section]
\newtheorem{lemma}[theorem]{Lemma}
\newtheorem{proposition}[theorem]{Proposition}
\newtheorem{definition}[theorem]{Definition}
\newtheorem{corollary}[theorem]{Corollary}
\newtheorem{remark}[theorem]{Remark}

\def\om{\omega}
\def\Om{\Omega}
\def\p{\partial}
\def\ep{\epsilon}
\def\de{\delta}

\def\S{{\Sigma}}
\def\<{\langle}
\def\>{\rangle}
\def\div{{\rm div}}
\def\na{\nabla}

\providecommand{\abs}[1]{\lvert#1\rvert}
\providecommand{\Abs}[1]{\left\lvert#1\right\rvert}

\newcommand{\mbN}{\mathbb{N}}

\newcommand{\mbP}{\mathbb{P}}

\newcommand{\mbS}{\mathbf{S}}

\newcommand{\mcB}{\mathcal{B}}
\newcommand{\mcC}{\mathcal{C}}

\newcommand{\mcH}{\mathcal{H}}

\newcommand{\mcT}{\mathcal{T}}

\newcommand{\mfQ}{\mathbf{Q}}

\newcommand{\mfW}{\mathbf{W}}

\newcommand{\msB}{\mathscr{B}}

\newcommand{\msH}{\mathscr{H}}

\newcommand{\msJ}{\mathscr{J}}
\newcommand{\msP}{\mathscr{P}}

\newcommand{\msU}{\mathscr{U}}

\newcommand{\rd}{{\rm d}}

\newcommand{\ra}{\rightarrow}

\newcommand{\eq}[1]{\begin{equation}\begin{alignedat}{2} #1 \end{alignedat}\end{equation}}

\numberwithin{equation}{section}

\begin{document}
	
\title[Liouville theorem]{Half-space Liouville-type theorems for minimal graphs with capillary boundary}

\date{\today}

\author[Wang]{Guofang Wang}
\address[G.W]{Mathematisches Institut\\
Universit\"at Freiburg\\
Ernst-Zermelo-Str.1\\
79104\\
\newline Freiburg\\ Germany}
\email{guofang.wang@math.uni-freiburg.de}

\author[Wei]{Wei Wei}
\address[W.W]{School of Mathematics\\ Nanjing University\\ 210093\\Nanjing\\ P.R. China}
\email{wei\_wei@nju.edu.cn}

\author[Zhang]{Xuwen Zhang}
	\address[X.Z]{Mathematisches Institut\\
		Universit\"at Freiburg\\
	Ernst-Zermelo-Str.1\\
		79104\\
  \newline Freiburg\\ Germany}
\email{xuwen.zhang@math.uni-freiburg.de}

\begin{abstract}
In this paper, we prove two Liouville-type theorems for capillary minimal graph over $\mathbb{R}^n_+$.
First, if $u$
has linear growth,
then for $n=2,3$ and for any $\theta\in(0,\pi)$, or $n\geq4$ and $\theta\in(\frac{\pi}6,\frac{5\pi}6)$, $u$ must be flat.
Second, if $u$ is one-sided bounded on $\mathbb{R}^n_+$, then for any $n$ and $\theta\in(0,\pi)$, $u$ must be flat.
The proofs build upon gradient estimates for the mean curvature equation over $\mathbb{R}^n_+$ with capillary boundary condition, which are based on carefully adapting the maximum principle to the capillary setting.

		\
		
		
	\end{abstract}

\maketitle
\tableofcontents 

\section{Introduction}
Let $\mathbb{R}^n_+=\{x\in\mathbb{R}^n:x_1>0\}$ be the upper half-space.
In this paper we study \textit{capillary minimal graphs}  over the half space, namely   graphs $\S=\{(x,u(x)):x\in\mathbb{R}^{n}_+\}$ of $u$, where $u:\mathbb{R}^n_+\to \mathbb{R}$  solves the minimal surface equation 
\eq{\label{eq:MSE-intro}
{\rm div}\left(\frac{Du}{\sqrt{1+\abs{Du}^2}}\right)
=0,\quad
\text{on }\mathbb{R}^n_+,
}
and satisfies the capillary boundary condition 
\eq{\label{eq:capillary-bdry-intro}
\left<\frac{(-Du(x),1)}{\sqrt{1+\abs{Du(x)}^2}},(1,0,\ldots,0)\right>
=\cos\theta,\quad\forall x\in\p\mathbb{R}^n_+,
}
for a fixed $\theta \in (0,\pi).$

The aim of this paper is to prove Liouville-type theorems for capillary minimal graphs over the half-space, which, roughly speaking, says that any capillary minimal graph with linear growth/one sided bounded on the half-space, must be flat. 
The motivation comes from a series of classical and recent progress on the study of minimal surface equation, which we briefly review.

For the minimal surface equation on the whole Euclidean space $\mathbb{R}^n$,
Moser \cite{Moser61} proved that if $\sup_{\mathbb{R}^n}\abs{Du}$ is bounded, then $u$ is flat, by using Harnack inequalities for uniformly elliptic equations. 
In 1969, Bombieri-De Giorgi-Miranda \cite{BDEGM69} established gradient estimates for solutions to the minimal surface equation on $\mathbb{R}^n$ (the $2$-dimension case was shown by Finn \cite{Finn54}), and then proved a Liouville theorem, which says that if in addition, the negative part of $u$ satisfies sub-linear growth (in particular, if $u$ is positive), then $u$ is a constant function.
 Caffarelli-Nirenberg-Spruck \cite{CNS88} extended the above Liouville theorem for $u$ with the assumption that $\abs{Du(x)}=o(\abs{x}^\frac12)$.
Later this was extended by Ecker-Huisekn \cite{EH90} for $u$ satisfying $\abs{Du(x)}=o\left(\sqrt{\abs{x}^2+\abs{u(x)}^2}\right)$. By contrast  Simon \cite{Simon89} constructed  a non-flat minimal graph, whose gradient satisfies $\abs{Du(x)}\leq C\abs{x}^{1+O(\frac1n)}$.
{\color{black}On the other hand, Bombieri-Giusti \cite{BG72} generalized Moser's result by assuming that \textit{only} $n-1$ partial derivatives of $u$ are bounded on $\mathbb{R}^n$, which is further extended by Farina in \cite{Farina15} by assuming that $n-1$ partial derivatives of $u$ are \textit{one-sided bounded}, and later in \cite{Farina18} by assuming that \textit{only} $n-7$ partial derivatives of $u$ are one-sided bounded.}
Very recently, there are many interesting results on Liouville Theorem for minimal graphs over a Riemannian manifold with certain curvature assumptions, see \cite{F-CS80,RSS13,DJX16,Ding21,Ding23,CMMR22,CMR23,CGMR24,Ding24}.

The celebrated Bernstein theorem, without any assumption on $u$, states that any minimal graph on $\mathbb{R}^n$ is flat if $n\leq7$.
This was proved by Bernstein \cite{Bernstein27} for $n=2$, by De Giorgi \cite{DeGiorgi65} for $n=3$, by Almgren \cite{Almgren66} for $n=4$, and by Simons \cite{Simons68} for $n=5,6,7$. See also \cite{BDEGG69} for the well-known counterexample for $n\geq8$ by Bombieri-De Giorgi-Giusti.
For its anisotropic counterpart,  the Bernstein theorem holds true, 
when $n=2$ by Jenkins \cite{Jenkins61} and $n=3$ by Simon \cite{Simon77}.
However,  this is no longer the case when $n\geq4$
by the recent results of Mooney and Mooney-Yang \cite{Mooney22,MY24}, in which they constructed anisotropic norms (or Minkowski norms) when $n\geq4$, such that a Bernstein-type result is not valid.
For norms obtained from a small $C^3$-perturbations of the Euclidean norm, Simon's result \cite{Simon77} shows that the Bernstein theorem holds true up to dimension $n=7$, see also a recent generalization by Du-Yang \cite{DY24}.

The last 5 years have witnessed progress on the study of Liouville theorem and Bernstein theorem for minimal surface equation with boundary condition.
Working in the half-space $\mathbb{R}^n_+$, Jiang-Wang-Zhu \cite{JWZ21} proved a Liouville theorem which says that, any solution to the minimal surface equation, having linear growth, with either Dirichlet boundary condition ($u=l$ on $\p\mathbb{R}^n_+$ where $l$ is a linear function), or Neumann boundary condition ($\p_{x_1}u=\lambda$ on $\mathbb{R}^n_+$ for some constant $\lambda\in\mathbb{R}$), must be affine.
{\color{black}See also Farina's results \cite{Farina22}, concerning homogeneous Dirichelt/Neumann boundary condition.}
For the above very rigid Dirichlet boundary condition, a Bernstein theorem was shown by
Edelen-Wang \cite{EW22}, which states that, beyond dimension restriction, any solution $u$ to the minimal surface equation on a convex domain $\Om\subset\mathbb{R}^n$, such that $u=l$ on $\p\Om$ where $l$ is a linear function, must be affine (see also \cite{JWZ23}). 
Recently an anisotropic generalization of this result is shown by Du-Mooney-Yang-Zhu \cite{DMYZ23}.

In terms of capillary boundary condition,
less result is known in the literature until the recent work by Hong-Saturnino \cite{HS23}, in which they showed a Bernstein theorem for stable capillary minimal surfaces (dimension $n=2$) in a Euclidean half-space, for general $\theta\in(0,\pi)$, see also De Masi-De Philippis \cite[Theorem 6.3]{MP21}.
Independently,
Li-Zhou-Zhu \cite{LZZ24} used the well-known Schoen-Simon-Yau \cite{SSY75} technique to obtain curvature estimate and consequently a Bernstein theorem for stable  capillary minimal hypersurfaces $\S^n\subset\mathbb{R}^{n+1}_+$ of dimension $2\leq n\leq5$, with no restriction on $\theta$ when $n=2$, while for $3\leq n\leq5$ certain restrictions on $\theta$.
Roughly speaking, $\theta$ can not be too far from $\frac\pi2$, and the range of $\theta$ decreases as the dimension $n$ increases.
Note also that in the free boundary case, curvature estimates for stable minimal hypersurfaces in general Riemannian manifolds were obtained by Guang-Li-Zhou \cite{GLZ20}.


In view of the above results, a natural question is to ask, whether or not a half-space Liouville theorem for minimal graph with capillary boundary holds.
The purpose of this paper is to address the problem, and we have the following results:

\begin{theorem}[Liouville-type Theorem $I$]\label{Thm:Liouville}
Let
$u$ be a smooth function on $\mathbb{R}^n_+$ and $\S$ be its corresponding graph, such that $\S$ is a capillary minimal graph,
and suppose that
 $u$ has linear growth on $\mathbb{R}^n_+$.
\begin{enumerate}
\item[($i$)] For $n=2,3$, $u$ is affine for general $\theta\in (0,\pi).$

    \item [($ii$)] For $n\ge 4$, if $\theta$ belongs to the range $\msU$, given by
    \eq{\label{condi:msU}
\msU
=\msU(n)=\left\{\theta\in(0,\pi): \abs{\cos\theta}^2<\frac{(3n-7)(n-1)}{4(n-2)^{2}}\right\},
}
    then $u$ is affine.

    \item [($iii$)] For any $n$ and general $\theta\in(0,\pi)$,
    there exists a positive constant $C_\theta$ depending only on $\theta$, with the following property:
    If $u$ is either bounded from above or from below by a linear function $L$ on $\mathbb{R}^n_+$, with $\abs{DL}\leq C_\theta$,
    then $u$ is affine.
\end{enumerate}
\end{theorem}
\begin{remark}
\normalfont
\begin{itemize}
    \item The range $\msU=\msU(n)$ in Item ($ii$) results from technical aspect, see Theorem \ref{Thm:gradient-estimate} for the detailed discussions.
    \item The constant $C_\theta$ in Item ($iii$) can be explicitly chosen as $\frac1{36}\frac {\abs{\cos\theta}(1-\sin\theta)}{(1+\frac{|\cos\theta |}{\sin\theta})}$.
    We point out that such a choice is the result of a technical analysis and is not sharp, see Theorem \ref{Thm:gradient-esti-u<=0} for details.
    By refining the analysis, one may obtain a slightly larger $C_\theta$.
    \item {\color{black}It is interesting to see that the small slope assumption in ($iii$) is also used in \cite{Ding24} to prove a Liouville-type theorem for entire minimal graphs with linear growth on manifolds with non-negative Ricci curvature, which
    can be removed if assuming the manifolds are of non-negative sectional curvature, see \cite[Corollary 10]{CGMR24}.}
\end{itemize}
\end{remark}
A direct consequence of ($iii$) is the following statement:
    For any $n$ and general $\theta\in(0,\pi)$, if $u$ has linear growth on $\mathbb{R}^n_+$ and is bounded from above or below by a constant function $L(x)\equiv C_L$ on $\mathbb{R}^n_+$, then $u$ is affine and must take the form
    \eq{\label{eq:u-symmetric}
u(x)
=-\cot\theta x_1+C.
    }
In this case, we can remove the linear growth assumption:
\begin{theorem}[Liouville-type Theorem $II$]\label{Thm:Liouville-positive}
For any $n$ and any $\theta\in (0,\pi)$,
let
$u$ be a smooth function on $\mathbb{R}^n_+$ and $\S$ be its corresponding graph, such that $\S$ is a capillary minimal graph.
If $u$ is one-sided bounded on $\mathbb{R}^n_+$, then $u$ is affine.

\end{theorem}

\begin{remark}
\normalfont
A
direct consequence of Theorem \ref{Thm:Liouville-positive} is the following non-existence result:
There is no smooth solution to the minimal surface equation \eqref{eq:MSE-intro} with capillary boundary condition \eqref{eq:capillary-bdry-intro}, if $\theta\in(0,\frac\pi2)$ and $u$ is bounded from below by a constant on $\mathbb{R}^n_+$; or $\theta\in(\frac\pi2,\pi)$ and $u$ is bounded from above by a constant on $\mathbb{R}^n_+$.
\end{remark}

The crucial step to prove the Liouville-type theorems is to show gradient estimates for minimal surface (mean curvature) equation on the half-space with capillary boundary condition.
Our two main estimates read as follows.
\begin{theorem}[Gradient estimate in terms of linear growth]\label{Thm:intro-gradient-linear-growth}
Let
$u$ be a smooth function on $\mathbb{R}^n_+$ and $\S$ be its corresponding graph, such that $\S$ is a capillary minimal graph.
Suppose that $u$ has linear growth on $\mathbb{R}^n_+$, namely, $\abs{u(x)}\leq C_{0}(1+\abs{x})$ for some constant $C_0>0$.
There exists a positive constant $\Lambda=\Lambda(n,\theta,C_0)$ with the following property:
If
\begin{enumerate}
    \item [($i$)] For $n=2,3$, and  general $\theta\in(0,\pi)$;
    \item [($ii$)] For $n\geq4$, $\theta$ belongs to the range $\msU$ defined in \eqref{condi:msU},
\end{enumerate}
then
\eq{
\sup_{\overline{\mathbb{R}_{+}^{n}}}\abs{D u}
\leq\Lambda.
}
\end{theorem}

\begin{theorem}[Gradient estimate for solutions with a sign]\label{Thm:intro-gradient-sign}
For any $n$ and any $\theta\in (0,\pi)$,
let $u$ be a smooth function on $\mathbb{R}^n_+$ and $\S$ be its corresponding graph, such that $\S$ is a capillary minimal graph.
There exists a positive constant $\widetilde\Lambda=\widetilde\Lambda(n,\theta)$ with the following property:
If $u$ is either bounded from above or from below by a constant function on $\mathbb{R}^n_+$,
then 
\eq{
    \sup_{\overline{\mathbb{R}^n_+}}\abs{Du}\leq\widetilde\Lambda.
    }
\end{theorem}

Our strategy to establish these estimates is to construct suitable auxiliary functions in the capillary settings and use the maximum principle.
Precisely, consider a solution $u$ to the following equation
\eq{
\begin{array}{rccl}
    (\de_{ij}-\frac{u_iu_j}{{1+\abs{Du}^2}})u_{ij}
    &=&0,\quad&\text{ in }\mathbb{R}^n_+,\\
    u_1&=&-\cos\theta\sqrt{1+\abs{Du}^2},\quad&\text{ on }\p\mathbb{R}^n_+,
\end{array}
}
where $u_i=\frac{\partial u}{\partial x_i}$, $u_{ij}= \frac{\partial^2 u}{\partial x_i \partial x_j}$.
Our main difficulty is to deal with boundary condition of such type.
In order to obtain gradient estimates, we first introduce a suitable, but unusual family of (ellipsoids) domains: for any $r>0$,
\eq{\label{Er}
E_r
\coloneqq&\left\{(x_1,x'):x_1>0, (x_1-\abs{\cos\theta} r)^2+\sin^2\theta\abs{x'}^2< r^2\right\},\\
E_{\theta, r}
\coloneqq&\left\{(x_1,x'):x_1>0, (x_1-\abs{\cos\theta} r)^2+\sin^2\theta\abs{x'}^2<\left(\frac{1+\abs{\cos\theta}}{2}r\right)^2\right\},
}
to replace round balls in the classical argument.
In $E_r$ we choose
a cut-off function $\psi$ defined as
\eq{\label{eq_a2}
\psi(x)= Q^2(x), \quad\hbox{with }
Q(x)\coloneqq1-\frac{(x_1-\abs{\cos\theta} r)^2+\sin^2\theta\abs{x'}^2}{r^2},
}
which will play a crucial role in our boundary estimates.
The auxiliary function that we construct to use the maximum principle is
\eq{\label{defn:G-intro}
G(x)=\varphi(u(x))\psi(x)\log v(x),
}
where $\varphi(u(x))=\frac{u(x)}{2M}+1$ with $M\coloneqq\sup_{E_r}\abs{u}+r$
and $v(x)\coloneqq\sqrt{1+\abs{Du(x)}^2}+\cos\theta u_1(x)$  is  the graphical capillary area element of the graph $\S\subset\mathbb{R}^{n+1}_+$.

With the help of the function $G$, 
we will show in Theorem \ref{Thm:gradient-estimate} that, for $n=2,3$ and general $\theta\in(0,\pi)$; for $n\geq4$ and $\theta$ belongs to $\msU$, there holds
\eq{\label{ineq:gradient-estimate-intro}
\sup_{\overline{E_{\theta, r}}}\abs{Du}
\leq \frac{1}{1-|\cos\theta|}\exp\left(C_1+C_2\frac{M}{r}+C_3\frac{M^2}{r^2}\right),
}
where $M=\sup_{E_r}\abs{u}+r$;
$C_1,C_2, C_3$ are positive constants depending only on $n,\theta$.
In particular, if $u$ has linear growth on $\mathbb{R}^n_+$, i.e., $\abs{u(x)}\leq C_0(1+\abs{x})$, this estimate implies the following global gradient estimate after sending $r\ra\infty$
\eq{\label{ineq:global-gradient-estimate-intro}
\sup_{\overline{\mathbb{R}^n_+}}\abs{Du}
\leq\Lambda=\Lambda(n,\theta,C_0).
}

The function $v$, used already in \cite{Uralcprime} and \cite{Gerhardt76}, has a nice property that $\left<\na v(x),\mu(x)\right>=0$ along the boundary (see Lemma \ref{good_property}, here $\na$ is the intrinsic gradient on the graph $\S$).
The cut-off function $\psi$, designed to couple with the capillary boundary condition, will be crucially used when carrying out a Hopf-type argument on $\p\mathbb{R}^n_+$, see {\bf Step 1}, especially \eqref{eq:a^i1-psi_i} in the proof of Theorem \ref{Thm:gradient-estimate}.
The choices of $v$ and $\psi$ enable us to overcome the difficulty resulting from the capillary boundary condition.
However, they bring new obstacles in the interior computations, mainly due to the appearance of $\cos\theta u_1$ in $v$.

{\bf Step 2} and {\bf Step 3} deal with the case that the maximum point $\max_{\overline{E}_r}G=G(z_0)$ is an interior point of $E_r$.
We will exploit the maximality condition at $z_0$:
\eq{\label{ineq:intro-max-interior}
0
\geq g^{ij}(z_0)(\log G)_{ij}
=g^{ij}\left(\frac{\varphi_{ij}}{\varphi}-\frac{\varphi_{i}}{\varphi}\frac{\varphi_{j}}{\varphi}+\frac{\psi_{ij}}{\psi}-\frac{\psi_{i}}{\psi}\frac{\psi_{j}}{\psi}+\frac{v_{ij}}{v\log v}-\frac{(1+\log v)v_{i}v_{j}}{(v\log v)^{2}}\right),
}
where $(g^{ij})$ corresponds to the inverse of the intrinsic metric of $\S$ as a hypersurface in $\mathbb{R}^{n+1}$.

The remained difficulty comes from 
the term $g^{ij}\left(\frac{v_{ij}}{v\log v}-\frac{(1+\log v)v_{i}v_{j}}{(v\log v)^{2}}\right)$, in which
the term $\cos\theta u_1$ in $v$ would result in a possible uncontrolled term if $\abs{\cos\theta}>\frac{\sqrt{3}}2$, see the estimate \eqref{ineq:C_i} in {\bf Step 2}.
Then in {\bf Step 3}, we use algebraic arguments (see the {\bf Claim} of {\bf Step 3}, and \eqref{defn:msB}) to control the possible bad term in \eqref{ineq:C_i}.
Thus, exploiting \eqref{ineq:intro-max-interior} we obtain gradient estimates for general $\theta\in(0,\pi)$ in low dimensions $n=2,3$,
while for $n\geq4$ we could
slightly push up the range of $\theta$, contributing to the set $\msU$ appearing in Theorem \ref{Thm:intro-gradient-linear-growth} ($ii$).

Different from the above proof, our strategy to approach the second main estimate (Theorem \ref{Thm:intro-gradient-sign}) is to use another auxiliary function (see \eqref{defn:G^ast-intro} below) to directly establish the boundary pointwise gradient estimate.
Once this crucial step is finished, we can then apply a classical argument to obtain the global gradient estimate on $\mathbb{R}^n_+$.

Let us briefly introduce how this crucial step works.
Consider for example $u<0$,
for any $p=(0,p')\in \partial \mathbb{R}^n_+$, fix a sufficiently large $r>0$, and define
\eq{\label{defn:G^ast-intro}
G^\ast(x)
=\varphi(u(x))\psi^\ast(x)\log v(x),\quad
\psi^\ast(x)
=\left(Q_{p}(x)+\frac{u(x)}{2N_\ast r}\right)^2,
}
where $\varphi(s)=\frac{s}{u(p)+r}+1$;
$Q_{p}(x)\coloneqq1-\frac{(x_1-\abs{\cos\theta} r)^2+\sin^2\theta\abs{x'-p'}^2}{r^2}$; and $N_\ast$ is a positive constant that needs to be suitably chosen (in fact, we can choose $N_\ast=\frac1{36}$).

The function $G^\ast$ is a modification of $G$ in \eqref{defn:G-intro}.
The advantages of such a choice are twofold: 
\textit{1.}
The function $Q_p$, designed to couple with the capillary boundary condition, is still sufficient for the boundary Hopf-type argument to work, despite the term $\frac{u}{2N_\ast r}$ in $\psi^\ast$ causes extra difficulty.
\textit{2.} In the interior computations, the term $\frac{u}{2N_\ast r}$ in $\psi^*$ now contributes a good term to control the possible bad terms caused by the term $\cos\theta u_1$ in $v$.
See Lemma \ref{Lem:gradient-estimate-bdry-u>0} also Theorem \ref{Thm:gradient-esti-u<=0} for the detailed discussions.

To end the introduction, we explain why our auxiliary functions are chosen in this way, by reviewing the history of gradient estimates for mean curvature equation.
First of all,
capillary boundary problem for mean curvature equation is a classical problem coming from physics, see Finn \cite{Finn86}.
Concerning the existence of a solution to the capillary boundary problem in bounded domain, 
the gradient estimates are essential and we refer to Spruck \cite{Spruck75}, Simon-Spruck \cite{SimonSpruck76}, Ural'tseva \cite{Uralcprime}, Gerhardt \cite{Gerhardt76}, Korevaar \cite{Korevaar1988}, Lieberman \cite{Lieberman13} for details. 
The function $v$ appearing in \eqref{defn:G-intro} is partly motivated by Ural'tseva \cite{Uralcprime} and  Gerhardt \cite{Gerhardt76} (see also Section \ref{Sec:2-2-1}), which is nowadays standard for capillary boundary problems, see e.g.,  Guan \cite{GuanB1994,GuanB1997}, Deng-Ma \cite{DM}, Gao-Lou-Xu \cite{GLX2024}, Lou-Yuan \cite{GLX2024,LYJMPA}.

The interior gradient estimates for mean curvature equation can be dated back to Bombieri-De Giorgi-Miranda \cite{BDEGM69}, which is based on integral methods.
See also Ladyzhenskaya-Ural’Tseva \cite{LU68}, Bombieri-Giusti \cite{BG72} and Trudinger \cite{Tru1972}.
Later, a new proof using maximum principle is provided in Korevaar's work \cite{Korevaar1986}, where he modified the cut-off function as $\left(1-\abs{x}^2-\frac{u(x)}{2u(0)}\right)^+$ for negative function $u$ in unit ball, see also \cite{EckerHuisken1991}. In \cite{Wang98}, by constructing  a new auxiliary function with a $\log$ term, X.-J. Wang gave another maximum principle proof.  
To study one-sided bounded $u$, for example $u<0$,
we combine the mentioned techniques together and modify them into the capillary setting, to construct the new auxiliary function
\eqref{defn:G^ast-intro}.

We believe that the technique developed in this paper  can be used to study a more general class of PDEs with capillary boundary condition.

\

\noindent{\it The rest of  the paper  is organized as follows.}
In Section \ref{Sec-2} we provide preliminaries on capillary minimal graphs
and study the graphical capillary area element $v$.
In Section \ref{Sec-3} we show gradient estimates for mean curvature equation on $\mathbb{R}^n_+$ with capillary boundary condition and linear growth assumption, which is Theorem \ref{Thm:gradient-estimate}.
Its refinements are presented in Subsection \ref{Sec-3-2}.
In Section \ref{Sec-4} we show gradient estimates for one-sided bounded solutions to minimal surface equation on $\mathbb{R}^n_+$ with capillary boundary condition.
In Section \ref{Sec-5} we prove Liouville-type theorems.

\

\noindent{\bf Acknowledgement.}
{\color{black}
We thank Prof. Alberto Farina and Prof. Luciano Mari for communications on their works related to this topic.
}
This work was carried out while
W. Wei was visiting University of Freiburg supported  by the Alexander von Humboldt research fellowship. She would like to thank Institute of Mathematics, University of Freiburg for its  hospitality. 
She was also
partially supported by NSFC (No. 12201288, 11771204) and BK20220755.
X. Zhang would like to thank Prof. Chao Xia for helpful discussions and constant encouragement.

\section{Preliminaries}\label{Sec-2}
\subsection{Notations}
Let $\{e_1,\ldots,e_{n+1}\}$ be the canonical basis of $\mathbb{R}^{n+1}$.
Consider the open half-space and its boundary
\eq{
\mathbb{R}^{n+1}_+
=\{x\in\mathbb{R}^{n+1}:x_1>0\},\quad
\p\mathbb{R}^{n+1}_+
=\{x\in\mathbb{R}^{n+1}:x_1=0\}.
}
Let $u$ be a smooth function defined on $\mathbb{R}^{n}_+$, and we denote its corresponding graph by
\eq{
\S\coloneqq\{(x,u(x)):x\in\mathbb{R}^{n}_+\},
}
which is a hypersurface in $\mathbb{R}^{n+1}_+$ with boundary $\p\S=\{(x,u(x)):x\in\p\mathbb{R}^n_+\}$.
The upwards-pointing unit normal of $\S\subset\mathbb{R}^{n+1}$, viewed as a vector field defined on $\mathbb{R}^n_+$, is given by
\eq{\label{eq:nu}
\nu(x)
=\frac{(-Du(x),1)}{\sqrt{1+\abs{Du(x)}^2}},\quad x\in\mathbb{R}^n_+,
}
where $Du(x)=(u_1(x),\ldots,u_n(x))$, and $u_i(x)\coloneqq\p_{x_i}u(x)$ for $i\in\{1,\ldots,n\}$.

If we write the map $\Phi:\mathbb{R}^n_+\ra\S, x\mapsto(x,u(x))$, then a basis of the tangent space of $\S$ is then given by $\{\tau_1,\ldots,\tau_n\}$, where
\eq{\label{defn:tau_i}
\tau_i(x)
=(\rd\Phi)_{x}(e_i)
=e_i+u_i(x)e_{n+1}.
}
The induced metric of $\S\subset\mathbb{R}^{n+1}$, denoted by $g$, and its inverse $g^{-1}$ are given by
\eq{
g_{ij}(x)
=\de_{ij}+u_i(x)u_j(x),\quad
g^{-1}_{ij}(x)
=\de_{ij}-\frac{u_i(x)u_j(x)}{{1+\abs{Du(x)}^2}},\quad x\in\mathbb{R}^n_+,\quad i,j\in\{1,\ldots,n\}.
}

\subsection{Capillary minimal graph}
\begin{definition}\label{Defn:capillary-graph}
\normalfont
Let $\theta\in(0,\pi)$,
$u$ a smooth function on $\mathbb{R}^n_+$ and $\S$  its corresponding graph.
Then $\S$ is called a \textit{capillary graph} in $\mathbb{R}^{n+1}_+$, if there holds
\eq{\label{condi:capillary-bdry-1}
\left<\nu(x),e_1\right>
=\cos\theta,\quad\forall x\in\p\mathbb{R}^n_+,
}
which is equivalent to
\eq{\label{condi:capillary-bdry-2}
u_1(x)
=-\cos\theta\sqrt{1+\abs{Du(x)}^2},\quad\forall x\in\p\mathbb{R}^n_+.
}
If in addition $H_\S\equiv0$,
then $\S$ is called a \textit{capillary minimal graph} in $\mathbb{R}^{n+1}_+$.
\end{definition}

Denote by $\mu$ the outer unit co-normal of $\p\S\subset \S$, which can also be viewed as a vector field defined on $\p\mathbb{R}^n_+$. It is easy to see that (see for example \cite{WW})
\begin{equation}\label{defn:mu1}
\mu(x)=\frac{-\sum_{i=2}^{n}u_i u_1 {\tau}_i+(1+|\bar Du|^2){\tau}_1}{\sqrt{(1+\abs{Du(x)}^2)(1+\abs{\bar Du(x)}^2)}}.
\end{equation}
It holds that for any function $f\in C^1(\overline{\mathbb{R}^n_+})$,
\eq{\label{defn:mu}
\langle\mu, (Df, f_{n+1})\rangle
=\frac{(1+\abs{\bar Du(x)}^2)\frac{\p f}{\p x_1}-u_1(x)\sum_{i=2}^nu_i(x)\frac{\p f}{\p x_i}}{\sqrt{(1+\abs{Du(x)}^2)(1+\abs{\bar Du(x)}^2)}},
}
where $f_{n+1}=0$ and $$\bar Du(x)=(u_2(x),\ldots,u_n(x))$$ is the Euclidean gradient of $u$ restricted to the $(n-1)$-plane $\p\mathbb{R}^n_+$.
It is clear that  the capillary boundary condition \eqref{condi:capillary-bdry-1} is the same as
\eq{
\left<\mu(x),e_1\right>
=-\sin\theta,\quad\forall x\in\p\mathbb{R}^n_+,
} which is in fact the definition of a capillary hypersurface.

\begin{figure}[H]
	\centering
	\includegraphics[width=14cm]{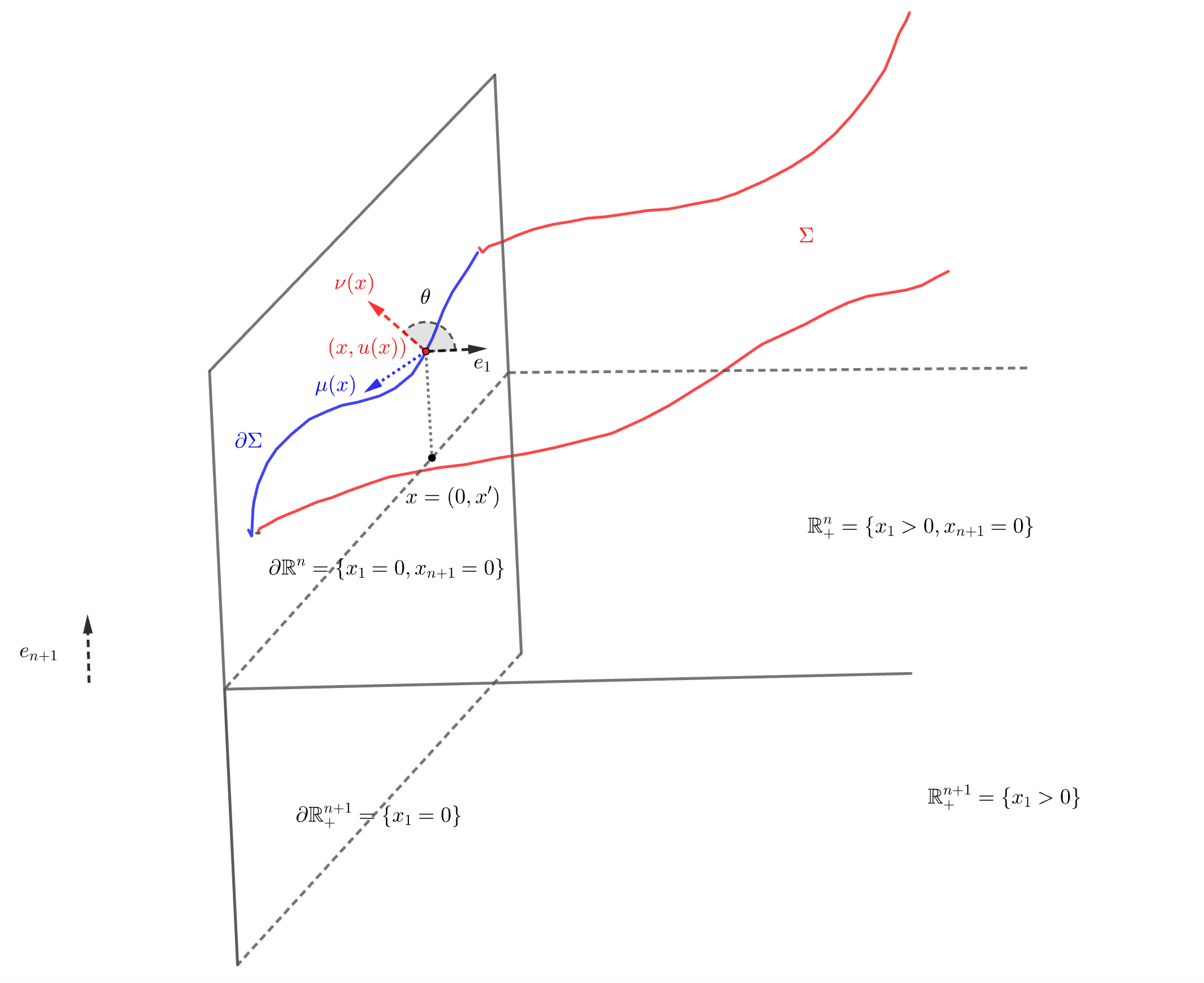}
	\caption{Capillary graph in $\mathbb{R}^{n+1}_+$}
	\label{Fig-1}
\end{figure}

\begin{remark}\label{Rem:theta>pi/2}
\normalfont
In the case that $\theta=\frac\pi2$, $\S$ as in Definition \ref{Defn:capillary-graph} has free boundary at $\p\mathbb{R}^{n+1}_+$, namely, $\S$ meets $\p\mathbb{R}^{n+1}_+$ orthogonally.
Therefore after reflecting it across the supporting hyperplane, we recover the classical minimal surface equation over the whole Euclidean space $\mathbb{R}^n$, which is fully understood.
\end{remark}

\subsection{Capillarity meets anisotropy}\label{Sec:2-2-1}
We use an idea of De Philippis-Maggi \cite{DePM15}, that half-space capillary problem is essentially an anisotropic problem, which can be seen by the definition of the following so-called capillary gauge function:
\begin{definition}
\normalfont
Given $\theta\in(0,\pi)$, the gauge function
\eq{\label{defn:F_theta}
F_\theta(\xi)
\coloneqq\abs{\xi}-\cos\theta\left<\xi,e_1\right>,\quad\xi\in\mathbb{R}^{n+1},
}
is called a \textit{capillary gauge}, which is a smooth function on $\mathbb{R}^{n+1}\setminus\{0\}$.
\end{definition}

For a capillary graph $\S$ as in Definition \ref{Defn:capillary-graph}, its \textit{$F_\theta$-surface energy} is defined as
\eq{\label{defn:capillary-area-functional}
\int_\S F_\theta(\nu(p))\rd\mcH^n(p).
}
By  integration by parts, one can see that the $F_\theta$-surface energy is exactly  the Gauss free energy with respect to the capillary angle $\theta$, namely, (see e.g., \cite[Proposition 3.3]{LXZ23})
\eq{\label{eq:F_theta-energy-capillary-energy}
\int_\S F_\theta(\nu(p))\rd\mcH^n(p)
=\abs{\S}-\cos\theta\abs{\p\Om\cap\p\mathbb{R}^{n+1}_+},
}
where $\Om$ is the domain delimited by $\S$ and $\p\mathbb{R}^{n+1}_+$, and $-\cos\theta\abs{\p\Om\cap\p\mathbb{R}^{n+1}_+}$ is the so-called \textit{wetting energy}.

Since $\S$ is the graph of $u$, by area formula and \eqref{eq:nu}, we could further write \eqref{defn:capillary-area-functional} as
\eq{
\int_{\mathbb{R}^n_+}F_\theta(\nu(x))\sqrt{1+\abs{Du(x)}^2}\rd x
=\int_{\mathbb{R}^n_+}\sqrt{1+\abs{Du(x)}^2}+\cos\theta u_1(x)\rd x\eqqcolon A_\theta(u).
}
The integrand
\eq{\label{defn:v}
v(x)
\coloneqq\sqrt{1+\abs{Du(x)}^2}+\cos\theta u_1(x),
}
is then called \textit{graphical capillary  area element} of the graph $\S$ with respect to $F_\theta$. It is clear that a critical point of the functional $A_\theta$ corresponds to a capillary minimal graph.

With this point of view, one may generalize the classical calibration argument to the capillary settings.
This is well-known to experts, but missing in the literature, therefore we include it in Appendix \ref{App-1}.
Consequently, one may obtain a half-space Bernstein-type theorem for capillary minimal graph, as a corollary of \cite{HS23,MP21,LZZ24}.
We point out that such an idea relies significantly on the fact that the domain (of $u$) is a half-space.
For general domains, it is well-known that the calibration argument works in the free boundary case.
Performing a contradiction argument in the same spirit as Guang-Li-Zhou \cite{GLZ20}, one can obtain a Bernstein-type theorem for free boundary minimal graphs, which
we include in Appendix \ref{App-2}.

Now we collect some useful facts concerning the function $v$.

\begin{lemma}[Positive lower bound]\label{Lem:v-lower-bound}
Let $v$ be given as above, then $v(x)\geq\sin\theta>0$ for any $x\in\overline{\mathbb{R}^n_+}$.
In particular, $v(x)=\sin\theta$ if and only if $u_1(x)=-\cot\theta$ and $\abs{\bar Du(x)}=0$.
\end{lemma}
\begin{proof}
By elementary computation, the one variable function $\sqrt{1+t^2}+\cos\theta t$ has minimal $\sin\theta$ on $\mathbb{R}$, and is attained only at $t=-\cot\theta$.
Thus
\eq{
v(x)
=\sqrt{1+\abs{Du(x)}^2}+\cos\theta u_1(x)
\geq\sqrt{1+(u_1(x))^2}+\cos\theta u_1(x)
\geq\sin\theta,
}
with equality holds if and only if $u_1(x)=-\cot\theta$ and $\abs{\bar Du(x)}=0$.
\end{proof}
The following lemma is essentially proved in \cite{Uralcprime,Gerhardt76}
\begin{lemma} \label{good_property}
Let $u$ be a smooth function on $\mathbb{R}^n_+$ and $\S$  its corresponding graph.
If $\S$ is a capillary graph in the sense of Definition \ref{Defn:capillary-graph}, then there holds pointwisely on $\p\mathbb{R}^n_+$
\eq{\label{eq:na-v-mu=0}
\left<\na v(x),\mu(x)\right>
=0,
}
where $\na$ is the intrinsic gradient of $\S$.
\end{lemma}
\begin{proof} Recall that $\mu$ is the outer unit co-normal of $\partial \Sigma \subset \Sigma$, 
given in \eqref{defn:mu}.
For ease of notation we put $\tilde a_i(x)=\frac{u_i(x)}{\sqrt{1+\abs{Du(x)}^2}}$.
Here we adopt the convention $u_i(x)=\p_{x_i}u(x)$ and $u_{ij}=\p^2_{x_ix_j}u$.

On $\p\mathbb{R}^n_+$, $\tilde a_1+\cos\theta\equiv0$ by \eqref{condi:capillary-bdry-2}, and hence at any $x\in\p\mathbb{R}^n_+$,
\eq{
\sum_{k=1}^n(\tilde a_k+\cos\theta\de_{1k})\frac{\p}{\p x_k}
=\sum_{k=2}^n\tilde a_k\frac{\p}{\p x_k},
}
which is a tangential operator on $\p\mathbb{R}^n_+$.
Hence, \eq{\label{eq:tanential derivative zero}
0
=\sum_{k=1}^n(\tilde a_k+\cos\theta\de_{1k})\frac{\p}{\p x_k}\left(\tilde a_1+\cos\theta\right).
}
Rewriting $\tilde a_i(x)=\frac{p_i(x)}{\sqrt{1+\abs{p(x)}^2}}$ with $p(x)=(p_1(x),\ldots,p_n(x))=Du(x)$,  by chain rules we have
\eq{\label{eq_a1}
\frac{\p \tilde a_i}{\p p_j}(x)
=\frac1{\sqrt{1+\abs{p(x)}^2}}(\de_{ij}-\frac{p_i(x)p_j(x)}{1+\abs{p(x)}^2})
=\frac1{\sqrt{1+\abs{p(x)}^2}}g^{ij}(x)
\eqqcolon \tilde a_{ij}(x),
}
and
\eq{
\frac{\p \tilde a_i}{\p x_l}(x)
=\frac{\p \tilde a_i}{\p p_j}(x)u_{jl}(x)
=\tilde a_{ij}(x)u_{jl}(x).
}
Together with \eqref{eq:tanential derivative zero}, it follows
\eq{\label{eq:a_ka_1ju_jk}
0
=\sum_{k=1}^n(\tilde a_k+\cos\theta\de_{1k})\frac{\p}{\p x_k}\left(\tilde a_1+\cos\theta\right)
=\sum_{k=1}^n(\tilde a_k+\cos\theta\de_{1k})(\tilde a_{1j}u_{jk}).
}
On the other hand, at any $x\in\p\mathbb{R}^n_+$, there holds (note that $\tilde a_i=(1+\abs{Du}^2)^{-\frac12}u_i$ by definition)
\eq{
\frac{\p v}{\p x_j}
=\sum_{k=1}^n(1+\abs{Du}^2)^{-\frac12}u_ku_{jk}+\cos\theta u_{j1}
=\sum_{k=1}^n(\tilde a_k+\cos\theta\de_{1k})u_{jk}.
}
It then follows (recall that $\tilde a_{ij}(x)=\frac{g^{ij}(x)}{\sqrt{1+\abs{Du(x)}^2}}$ by \eqref{eq_a1})
\eq{
(1+\abs{Du}^2)^{-\frac12}g^{1j}\partial_{x_j}v
=\tilde a_{1j}\left(\sum_{k=1}^n(\tilde a_k+\cos\theta\de_{1k})u_{jk}\right)
=0,
}
where we have used \eqref{eq:a_ka_1ju_jk} for the last equality.
This implies that $g^{1j}v_j\equiv0$ on $\p\mathbb{R}^n_+$.

Finally, taking \eqref{defn:mu} into account, we thus find
\eq{
0
=g^{1j}\partial_{x_j}v
=&(\de_{1j}-\frac{u_1u_j}{1+\abs{Du}^2})\partial_{x_j}v
=\partial_{x_1}v-\sum_{i=1}^n\frac{u_1u_i}{1+\abs{Du}^2}\partial_{x_i}v\\
=&\frac{1+\abs{\bar Du}^2}{1+\abs{Du}^2}\partial_{x_1}v-\sum_{i=2}^n\frac{u_1u_i\partial_{x_i}v}{1+\abs{Du}^2}
=\frac{\mu(v)}{\sqrt{1+\abs{Du}^2}}\sqrt{1+\abs{\bar Du}^2},
}
namely, $\mu(v)\equiv0$ on $\p\mathbb{R}^n_+$.
The assertion then follows since on $\p\mathbb{R}^n_+$,
\eq{
\left<\na v,\mu\right>
=\mu(v).
}
\end{proof}

\section{Gradient estimates}\label{Sec-3}

Consider the mean curvature equation on $\mathbb{R}^n_+$:
\eq{\label{eq:MSE}
a^{ij}u_{ij}
\coloneqq(W^{2}\delta_{ij}-u_{i}u_{j})u_{ij}=HW^3,
}
where $W=\sqrt{1+|Du|^{2}}$,
with capillary boundary condition: $\frac{u_{1}}{W}=-\cos\theta$ on $\p\mathbb{R}^n_+$.
Or equivalently, the graph $\S$ corresponding to $u$ is a capillary graph in $\mathbb{R}^n_+$ in the sense of Definition \ref{Defn:capillary-graph}.
$H(x)$ denotes the prescribed mean curvature function.

\subsection{Gradient estimates for mean curvature equation}\label{Sec-3-1}

\begin{lemma}[Cut-off functions]\label{Lem:cut-off}
Let $\theta\in(0,\pi)$.
For any $r>0$, define the (ellipsoids) sets on $\mathbb{R}^n_+$:
\eq{
E_r
\coloneqq&\left\{(x_1,x'):x_1>0, (x_1-\abs{\cos\theta} r)^2+\sin^2\theta\abs{x'}^2< r^2\right\},\\
E_{\theta, r}
\coloneqq&\left\{(x_1,x'):x_1>0, (x_1-\abs{\cos\theta} r)^2+\sin^2\theta\abs{x'}^2<\left(\frac{1+\abs{\cos\theta}}{2}r\right)^2\right\},
}
then $E_{\theta,r}\subset E_r$, with $\lim_{r\rightarrow \infty}E_r=\mathbb{R}^n_+$, and $\lim_{r\rightarrow\infty}E_{\theta, r}=\mathbb{R}^n_+$.

The cut-off function $\psi$ defined as
\eq{
\psi(x)
=\left(1-\frac{(x_1-\abs{\cos\theta} r)^2+\sin^2\theta\abs{x'}^2}{r^2}\right)^2,
}
satisfies (write for simplicity $\p_{rel}E_r=\overline{\p E_r\cap\mathbb{R}^n_+}$ as the relative boundary of $E_r$ in $\mathbb{R}^n_+$)
\eq{
\left(1-\frac{(1+\abs{\cos\theta})^2}4\right)^2
<\psi\leq 1\text{ in }E_{\theta,r},\quad
\psi\equiv0\text{ on }\p_{rel} E_r,\quad
\frac{\p\psi}{\p x_1}
=4\psi^{\frac12}\frac{\abs{\cos\theta}}r\text{ on }\p\mathbb{R}^n_+.
}
Moreover, there exists a positive constant $c_{n,\theta}$, depending only on $n,\theta$, such that in $E_r$ there hold
\eq{
\abs{D\psi}
\leq 4\frac{\psi^{\frac12}}r,\quad
\abs{D^2\psi}
\leq c_{n,\theta}\frac1{r^2}.
}
\end{lemma}
\begin{proof}
Direct computations show that on $\mathbb{R}^n_+$:
\eq{
\frac{\p\psi}{\p x_1}
=&4\left(1-\frac{(x_1-\abs{\cos\theta} r)^2+\sin^2\theta\abs{x'}^2}{r^2}\right)\left(-\frac{1}{r^2}(x_1-\abs{\cos\theta} r)\right),\\
\frac{\p\psi}{\p x_i}
=&4\left(1-\frac{(x_1-\abs{\cos\theta} r)^2+\sin^2\theta\abs{x'}^2}{r^2}\right)\left(-\frac{\sin^2\theta x_i}{r^2}\right),\quad i\in\{2,\ldots,n\},
}
the assertions then follow.
\end{proof}

\begin{theorem}\label{Thm:gradient-estimate}
Let $\theta\in(0,\pi)$,
let $u$ be a $C^2$-solution of the mean curvature equation \eqref{eq:MSE}, such that its graph $\S$ is a capillary graph in the sense of Definition \ref{Defn:capillary-graph}.
\begin{enumerate}
    \item [($i$)] Assume that $\abs{H}+\abs{DH}\leq\mathbf{C}_H$ on $\mathbb{R}^n_+$ for some positive constant $\mathbf{C}_H$.
    If $\abs{\cos\theta}<\frac{\sqrt{3}}2$,
then for any $r>0$, there holds
\eq{\label{eq:Du-C_1-C_2-C_3-i}
\sup_{\overline{E_{\theta, r}}}\abs{Du}
\leq \frac{1}{1-|\cos\theta|}\exp\left(C_1+C_2\frac{M}{r}+C_3\frac{M^2}{r^2}\right),
}
where $M=\sup_{E_r}\abs{u(x)}+r$, $C_1$ depends only on $n,\theta,\mathbf{C}_H$ and $ M$, $C_2$ and $C_3$ depend only on $n,\theta,\mathbf{C}_H$.
    \item [($ii$)] Assume that $H\equiv0$ (i.e., $u$ solves the minimal surface equation).
    If $n=2,3$ and for general $\theta\in(0,\pi)$; or for $n\geq4$ and $\theta$ belongs to the range $\msU$, where $\msU$ was defined by \eqref{condi:msU}, 
then for any $r>0$, there holds
\eq{\label{eq:Du-C_1-C_2-C_3-ii}
\sup_{\overline{E_{\theta, r}}}\abs{Du}
\leq \frac{1}{1-|\cos\theta|}\exp\left(C_1+C_2\frac{M}{r}+C_3\frac{M^2}{r^2}\right),
}
 where $M=\sup_{E_r}\abs{u(x)}+r$ and $C_1,C_2$ and $C_3$ depend only on $n$ and $\theta$.
Moreover, suppose that $u$ has linear growth, namely, $\abs{u(x)}\leq C_{0}(1+\abs{x})$ for some constant $C_0>0$, we have
\eq{\label{conclu:Du-bdd}
\sup_{\overline{\mathbb{R}_{+}^{n}}}\abs{D u}
\leq\Lambda,
}
where $\Lambda$ depends only on $n, \theta, C_{0}$.
\end{enumerate}

\end{theorem}
\begin{proof}[Proof of Theorem \ref{Thm:gradient-estimate} (Theorem \ref{Thm:intro-gradient-linear-growth})]

Recalling Remark \ref{Rem:theta>pi/2},
in the following we only consider those $\theta\in(0,\pi)\setminus\{\frac\pi2\}$.

We continue to use the notations in Lemma \ref{Lem:cut-off}.
For any fixed  $r>0$, we consider the function:
\eq{
G(x)=\varphi(u(x))\psi(x)\log v(x),
}
where $\varphi(s)=\frac{s}{2M}+1$ with $M=\sup_{E_{r}}\abs{u}+r$, and (recall \eqref{defn:v})
\eq{\label{eq:v}
v
=W+\cos\theta u_1.
}
Let $z_0\in\overline{E_r}$ be such that
\eq{
\max_{\overline{E_r}}G
=G(z_0).
}

Observe that if $\sup_{\overline{E_r}}\abs{Du}$ is sufficiently large, then $\sup_{\overline{E_r}}v$ is also sufficiently large by $v\ge (1-|\cos\theta|)\sqrt{1+|Du|^2}$.
Note that on $\overline{E_{\theta,r}}\subset\overline{E_r}$, by Lemma \ref{Lem:cut-off}
\begin{equation}\label{Basic properties of auxiliary functions}
\frac12<\varphi(u(x))<\frac32,\quad
\left(1-\frac{(1+\abs{\cos\theta})^2}4\right)^2\leq\psi\leq1.
\end{equation}
Hence 
we may assume that $G(z_0)$ is positive and sufficiently large, otherwise there is nothing to prove.
In this case, $z_0\notin\p_{rel} E_r$ by construction of the cut-off function $\psi$.
Also, we may assume
\eq{\label{ineq:psi-Du-(z_0)}
\psi(z_0)\abs{Du(z_0)}\geq1.
}

\noindent{\bf Step 1. 
We deal with the case that the maximum point $z_0\in\p E_r\setminus\p_{rel}E_r$.
}

Thanks to \eqref{condi:capillary-bdry-2} there holds
\eq{\label{eq:a^i1-u_i}
a^{i1}u_{i}
=W^{2}u_{1}-\sum u_{i}^{2}u_{1}=u_{1}
=
-\cos\theta W\quad \text{on}\quad  \p\mathbb{R}_{+}^{n}.
}
From \eqref{condi:capillary-bdry-2}, we have
$\cos^2\theta(1+\abs{\bar D u}^2+u_1^2)=u_1^2$,
and hence
\eq{\label{eq:bar-Du-u_1-bdry}
\abs{\cos\theta}\sqrt{(1+\abs{\bar Du}^2)}
=\sin\theta\abs{u_1}.
}
Then, by Lemma \ref{Lem:cut-off} and the fact  that $\abs{x'}\le r$ on $\p E_r\setminus\p_{rel}E_r$ (on which $x_1=0$), we estimate
\eq{\label{eq:a^i1-psi_i}
a^{i1}\frac{\psi_i}{\psi}
&=\frac{W^2}{\psi}\left(\psi_1-\frac{\sum_{i=1}^nu_iu_1\psi_i}{W^2}\right)\\
&=\frac{W^2}{\psi}\left(\psi_1 \frac{1+\abs{\bar Du}^2}{W^2}+\frac{\sum_{i=2}^nu_iu_1\frac{4\sin^2\theta x_i}{r^2}\psi^{\frac12}}{W^2}\right)\\
&=\frac{1}{r\psi^{\frac12}}\left(4\abs{\cos\theta}(1+\abs{\bar Du}^2)+4\sin^2\theta\sum_{i=2}^nu_iu_1\frac{x_i}{r}\right)\\
&\ge\frac{1}{r\psi^{\frac12}}\left(4\abs{\cos\theta}(1+\abs{\bar Du}^2)-4\sin^2\theta\abs{u_1}\abs{\bar Du}\right)\\
&>\frac{1}{r\psi^{\frac12}}\left(4\abs{\cos\theta}\sqrt{(1+\abs{\bar Du}^2)}\abs{\bar D u}-4\sin^2\theta\abs{u_1}\abs{\bar Du}\right)\\
&=\frac4{r\psi^{\frac12}}\abs{\cos\theta}(1-\sin\theta)\sqrt{(1+\abs{\bar Du}^2)}\abs{\bar D u}
>0.
}
Thus, at $z_0$, we have (recall \eqref{eq:na-v-mu=0}, we have $a^{ij}v_{i}(x_{1})_{j}=0$)
\eq{
0\ge a^{ij}(\log G)_{i}(x_{1})_{j}
=&a^{i1}\left(\frac{\varphi'u_{i}}{\varphi}+\frac{v_{i}}{v\log v}+\frac{\psi_i}{\psi}\right)
=a^{i1}\frac{\psi_i}{\psi}+a^{i1}\frac{\varphi'}{\varphi}u_{i}\\
\overset{\eqref{eq:a^i1-u_i},\eqref{eq:a^i1-psi_i}}{\ge}&\frac4{r\psi^{\frac12}}\abs{\cos\theta}(1-\sin\theta)\sqrt{(1+\abs{\bar Du}^2)}\abs{\bar D u}-\frac{1}{2M\varphi}(\abs{\cos\theta} W).
}
Recalling \eqref{eq:bar-Du-u_1-bdry}, we get $\abs{Du(z_0)}\leq C(\theta)\frac{r\psi^\frac12}{M}\leq C(\theta)\frac{r}{r+\sup_{E_r}\abs{u}}\leq C(\theta)$.
Since $z_0$ is the maximum point of $G$, we thus find
\eq{\label{esti:Thm3.2-bdry-maximum}
C(\theta)\log v(x)
\overset{\eqref{Basic properties of auxiliary functions}}{\leq}\varphi(u(x))\psi(x)\log v(x)
=G(x)
\leq G(z_0)
=\varphi(u(z_0))\psi(z_0)\log v(z_0)
\leq C(\theta),\quad\forall
x\in E_{\theta, r}.
}
Namely, in this case we have the required estimate.

\

\noindent{\bf Step 2.
We prove \eqref{eq:Du-C_1-C_2-C_3-i}, and also \eqref{eq:Du-C_1-C_2-C_3-ii} for the case $\abs{\cos\theta}<\frac{\sqrt{3}}2$.
}

By {\bf Step 1}, we just have to consider the case that the maximum point $z_0\in E_{r}$.
By maximality, at $z_0$,  
$(\log G)_{i}=0$,
and the matrix $\left(\log G\right)_{ij}\leq0$.
Namely, (write $\varphi_i=\varphi'(u)u_i$, $\varphi_{ij}$ is understood similarly)
\eq{\label{eq:logG_i}
0
=(\log G)_{i}
=\frac{\varphi_{i}}{\varphi}+\frac{\psi_{i}}{\psi}+\frac{v_{i}}{v\log v},
}
and 
\eq{\label{ineq:maximun-principle}
0\geq a^{ij}(\log G)_{ij}
=a^{ij}\left(\frac{\varphi_{ij}}{\varphi}-\frac{\varphi_{i}}{\varphi}\frac{\varphi_{j}}{\varphi}+\frac{\psi_{ij}}{\psi}-\frac{\psi_{i}}{\psi}\frac{\psi_{j}}{\psi}+\frac{v_{ij}}{v\log v}-\frac{(1+\log v)v_{i}v_{j}}{(v\log v)^{2}}\right).
}

Now after  a suitable rotation of coordinates (denote by $\{\tilde e_i\}$ the new coordinate basis), we assume that $\abs{Du(z_0)}=u_{n}(z_0)$, and $\{u_{ij}(z_0)\}_{1\leq i,j\leq n-1}$ is a diagonal matrix.
Note that after this rotation,
if $\cos\theta\in(0,\frac\pi2)$ we put $\left<Du,e_1\right>=\left<e_1,\tilde e_i\right>\left<Du,\tilde e_i\right>\eqqcolon b_iu_i$,
where $\{b_{i}\}_{i=1}^n$ are constants with $\sum_{i=1}^nb_i^2=1$;
if $\cos\theta\in[\frac\pi2,\pi)$ we put $-\left<Du,e_1\right>=\left<-e_1,\tilde e_i\right>\left<Du,\tilde e_i\right>\eqqcolon b_iu_i$,
where $\{b_{i}\}_{i=1}^n$ are constants with $\sum_{i=1}^nb_i^2=1$.
Therefore now we should write \eqref{eq:v} as 
\eq{\label{eq:v-z_0}
v
=W+\abs{\cos\theta}\sum_{k=1}^nu_kb_k
=W+\abs{\cos\theta}u_nb_n.
}

The following computations are carried out at $z_0$.
Note that $u_{i}\left(z_0\right)=0$ for $1\le i\le n-1$, thus 
\eq{\label{eq:a^ij-z_0}
a^{nn}=1;\quad
a^{ii}=1+u_{n}^{2}=W^2,\quad 1\leq i\leq n-1;\text{ and }
a^{ij}=0,\quad1\leq i\neq j\leq n.
}
Thus at $z_0$, \eqref{eq:MSE} reads
\eq{\label{eq:MSE-z_0}
-\sum_{i=1}^{n-1}u_{ii}
=\frac1{1+u_n^2}\left(u_{nn}-{HW^3}\right).
}
To proceed we use
direct computation to obtain
\eq{\label{eq:v_i}
v_{i}
=\frac{u_{n}u_{ni}}{W}+\abs{\cos\theta} u_{ki}b_{k}.
}
Back to \eqref{eq:logG_i}, we get
\eq{
\frac{u_{n}u_{ni}}{W}+\abs{\cos\theta}u_{ki}b_{k}
=-v\log v\left(\frac{\varphi_{i}}{\varphi}+\frac{\psi_{i}}{\psi}\right).
}
Put for simplicity $A=\frac{u_{n}}{W}+\abs{\cos\theta}b_{n}$.
For $i=n$, the above equality reads
\eq{
u_{nn}
=-\frac{\abs{\cos\theta}}A\sum_{k=1}^{n-1}u_{kn}b_{k}-\frac1A\left(\frac{\varphi_{n}}{\varphi}+\frac{\psi_{n}}{\psi}\right)v\log v.
}

For $i=1,\cdots,n-1$, note that $\varphi_i(z_0)=\varphi'(u)u_i(z_0)=0$, $u_{ik}(z_0)=0$ for $1\le i\neq k\le n-1$, it holds that 
\eq{\label{eq:u_ni}
u_{ni}
=-\frac{\abs{\cos\theta}}{A}b_{i}u_{ii}-\frac{\psi_{i}}{A\psi}v\log v,
}
which in turn gives
\eq{\label{eq:u_nn}
u_{nn}
&=\frac{\cos^2\theta}{A^{2}}\sum_{i=1}^{n-1}b_{i}^{2}u_{ii}-\frac{v\log v}{A}\left(\frac{\varphi_{n}}{\varphi}+\frac{\psi_{n}}{\psi}\right)+\frac{\abs{\cos\theta}v\log v\sum_{k=1}^{n-1}b_{k}\psi_{k}}{A^{2}\psi}.
}

\noindent{\bf Step 2.1. We bound $u_{nn}^2$ from below.
}

Note that $\abs{\cos\theta}\sum_{k=1}^{n-1}\abs{b_k}\leq n$.
For $G(z_0)$ sufficiently large depending on $\theta$, we may assume that
\eq{\label{defn:A}
\frac{1-\abs{\cos\theta}}2
<A
=\frac{u_n}W+\abs{\cos\theta}b_n
<1+\abs{\cos\theta},
}
and
we may also assume that
\eq{\label{ineq:ep_3-varphi}
\frac18{\frac{\varphi_{n}}{\varphi}}
>C(n,\theta)\frac{\abs{D\psi}}{\psi}
\geq\frac{\abs{\psi_n}}{\psi}+\frac{\abs{\cos\theta}\sum_{k=1}^{n-1}\abs{b_k\psi_k}}{A\psi}.
}
Otherwise,
as $\varphi_n=\varphi'(u)\abs{Du(z_0)}$ we must have $\psi^{\frac12}\abs{Du}_{\mid_{z_0}}<\frac{C(n,\theta)M}{r}$, and hence by \eqref{Basic properties of auxiliary functions}
\eq{\label{ineq:sup-v-ep_3}
\frac{(1-|\cos\theta|)^2}{4}\sup_{\overline{E_{\theta,r}}}\log v
\leq\sup_{\overline{ E_{\theta,r}}}G
\leq\sup_{\overline{E_r}}\left(\varphi(u)\psi\log v\right)
\leq 3\psi^{\frac12}\abs{Du}_{\mid_{z_0}}<\frac{C(n,\theta)M}{r},
}
which gives the required estimates in ($i$) and 
($ii$). 

With \eqref{ineq:ep_3-varphi}, we could in turn go back to \eqref{eq:u_nn} and use Cauchy inequality to get
\eq{\label{eq:lower bound og unn2}
u_{nn}^2
&\geq
\frac34\frac{1}{A^2}u_{n}^{2}\left(v\log v\right)^{2}\left(\frac{\varphi'}{\varphi}\right)^{2}-C(\theta)\sum_{i=1}^{n-1}u_{ii}^{2}.
}

\noindent{\bf Step 2.2. We estimate the last two terms appearing in \eqref{ineq:maximun-principle}.}

Differentiating equation \eqref{eq:MSE} gives $a^{ij}u_{ijk}+\frac{\p a^{ij}}{\p p_{l}}u_{lk}u_{ij}={(HW^3)_k}$,
where (recall that $\abs{Du(z_0)}=u_n(z_0)$)
\eq{
\frac{\p a^{ij}}{\p p_l}(Du)\mid_{z_0}
=2u_n\de_{ln}\de_{ij}-\de_{il}\de_{jn}u_n-\de_{jl}\de_{in}u_n\mid_{z_0}.
}
Hence at $z_0$ we have (recall that $\left(u_{ij}(z_0)\right)_{1\leq i,j\leq n-1}$ is diagonal)
\eq{
\frac{\p a^{ij}}{\p p_{l}}u_{lk}u_{ij}
=&\left(2u_n\de_{ln}\de_{ij}-\de_{il}\de_{jn}u_n-\de_{jl}\de_{in}u_n\right)u_{lk}u_{ij}\\
=&2u_nu_{nk}\sum_{i=1}^{n-1}u_{ii}+2u_nu_{nk}u_{nn}-2u_n\sum_{i=1}^{n-1}u_{in}u_{ik}-2u_nu_{nn}u_{nk}\\
\overset{\eqref{eq:MSE-z_0}}{=}&\frac{-2u_nu_{nn}u_{nk}+2u_nu_{nk}HW^3}{1+u_n^2}-2u_n\sum_{i=1}^{n-1}u_{in}u_{ik},
}
which implies
\eq{\label{eq:a^ij-u_ijk}
a^{ij}u_{ijk}
=\frac{2u_{n}u_{nn}u_{nk}}{1+u_{n}^{2}}+2u_{n}\sum_{i=1}^{n-1}u_{in}u_{ik}+{H_kW^3+HWu_nu_{nk}}.
}
In particular, for $k=n$ we have
\eq{\label{eq:a^ij-u_ijn}
a^{ij}u_{ijn}
=\frac{2u_{n}u_{nn}^{2}}{1+u_{n}^{2}}+2u_{n}\sum_{i=1}^{n-1}u_{ni}^{2}+{H_nW^3+HWu_nu_{nn}}.
}

Differentiating $v=W+\abs{\cos\theta}\sum_{k=1}^nu_kb_k$ twice, we obtain
\eq{
a^{ij}v_{ij}
=\underbrace{\frac{u_{n}}{W} a^{ij}u_{ijn}+\abs{\cos\theta} a^{ij}u_{ijk}b_{k}}_{\coloneqq I_{11}}
+\frac{1}{W^{3}}a^{ij}u_{ni}u_{nj}+
\frac1W\sum_{k=1}^{n-1}a^{ij}u_{ki}u_{kj}.
}
Using \eqref{eq:a^ij-u_ijk} and \eqref{eq:a^ij-u_ijn}, we could further write
\eq{
I_{11}
=&\frac{2u_{n}u_{nn}}{W^2}\left(\frac{u_{n}}{W}u_{nn}+\abs{\cos\theta}\sum_{k=1}^{n}u_{nk}b_{k}\right)
+2u_{n}\sum_{i=1}^{n-1}\left(\frac{u_n}Wu_{ni}+\abs{\cos\theta}\sum_{k=1}^{n}u_{ik}b_{k}\right)u_{ni}\\
&+\frac{u_n}W\left({\underbrace{H_nW^3+HWu_nu_{nn}}_{\coloneqq \mcT_n}}\right)+\abs{\cos\theta}\sum_{k=1}^nb_k\left({\underbrace{H_kW^3+HWu_nu_{nk}}_{\coloneqq\msH_k}}\right),
}
and in turn
\eq{\label{eq:formula 1}
a^{ij}v_{ij}
=&\frac{2u_{n}^{2}u_{nn}^{2}}{W^{3}}+\frac{2\abs{\cos\theta}u_{n}u_{nn}^{2}b_n}{W^{2}}+\frac{2u_{n}^{2}\sum_{i=1}^{n-1}u_{ni}^{2}}{W}+\frac{2\abs{\cos\theta}u_{n}u_{nn}}{W^{2}}\sum_{k=1}^{n-1}u_{nk}b_{k}\\
&+2u_{n}\abs{\cos\theta}\sum_{i
=1}^{n-1}u_{in}^{2}b_{n}+2u_{n}\abs{\cos\theta}\sum_{i=1}^{n-1}u_{ii}b_{i}u_{ni}\\
&+\frac{1}{W^{3}}\left(u_{nn}^{2}+W^{2}\sum_{i=1}^{n-1}u_{ni}^{2}\right)+\sum_{k=1}^{n-1}Wu_{kk}^{2}+\frac{1}{W}\sum_{k=1}^{n-1}u_{kn}^{2}
{+\frac{u_n}W\mcT_n+\abs{\cos\theta}\sum_{k=1}^nb_k\msH_k}.
}

On the other hand, we use
\eqref{eq:v_i} and \eqref{eq:a^ij-z_0} to compute
\eq{\label{eq:aijvivj upperbpund}
a^{ij}v_{i}v_{j}
=&a^{ij}\left(\frac{u_{n}u_{ni}}{W}+\abs{\cos\theta} u_{ki}b_{k}\right)\left(\frac{u_{n}u_{nj}}{W}+\abs{\cos\theta} u_{lj}b_{l}\right)\\
=&\frac{u_{n}^{2}}{W^{2}}a^{ij}u_{ni}u_{nj}+2\abs{\cos\theta} u_{ki}b_{k}a^{ii}\frac{u_{n}u_{ni}}{W}+a^{ij}\abs{\cos\theta}^2u_{ki}b_{k}u_{lj}b_{l}\\
=&\frac{u_{n}^{2}}{W^{2}}(u_{nn}^{2}+W^{2}\sum_{i=1}^{n-1}u_{ni}^{2})+\frac{2\abs{\cos\theta}u_{nn}^{2}b_{n}u_{n}}{W}+\frac{2\abs{\cos\theta}\sum_{k=1}^{n-1}u_{kn}b_{k}u_{n}u_{nn}}{W}\\
&+\sum_{i=1}^{n-1}2\abs{\cos\theta} Wu_{ni}b_{n}u_{n}u_{ni}+\sum_{i=1}^{n-1}2\abs{\cos\theta}Wu_{ii}b_{i}u_{n}u_{ni}\\
&+\abs{\cos\theta}^{2}{\bigg(}b_{n}^{2}u_{nn}^{2}+2b_{n}u_{nn}\sum_{l=1}^{n-1}b_{l}u_{nl}+(\sum_{k=1}^{n-1}b_{k}u_{kn})^{2}\\
 & +W^{2}\sum_{i=1}^{n-1}u_{ni}^{2}b_{n}^{2}+W^{2}\sum_{i=1}^{n-1}u_{ii}^{2}b_{i}^{2}+2W^{2}\sum_{i=1}^{n-1}u_{ni}u_{ii}b_{n}b_{i}{\bigg)}.
}
By \eqref{eq:formula 1} and \eqref{eq:aijvivj upperbpund}, we obtain
\eq{\label{eq:J expression}
a^{ij}v_{ij}-\frac{(1+\log v)}{(v\log v)}a^{ij}v_{i}v_{j}
\coloneqq\underbrace{\msJ_1+\msJ_2+\msJ_3+\msJ_4}_{\coloneqq\msJ}{+
\left(\frac{u_n}W\mcT_n+\abs{\cos\theta}\sum_{k=1}^nb_k\msH_k\right)},
}
where $\msJ_i$ ($i=1,..,4$) are defined by
\eq{\label{defn:J_1}
\msJ_1
\coloneqq u_{nn}^{2}\left(\frac{2u_{n}^{2}}{W^{3}}+\frac{2\abs{\cos\theta}u_{n}b_{n}}{W^{2}}+\frac{{1}}{W^{3}}-\frac{(1+\log v)}{v\log v}\left(\frac{u_{n}^{2}}{W^{2}}+\abs{\cos\theta}^{2}b_{n}^{2}+\frac{2\abs{\cos\theta}u_nb_{n}}{W}\right)\right),
}
\eq{\label{defn:J_2}
\msJ_2
&\coloneqq\sum_{i=1}^{n-1}u_{ni}^{2}\left(\frac{2u_{n}^{2}}{W}+2u_{n}\abs{\cos\theta}b_{n}+\frac{2}{W}-\frac{(1+\log v)}{(v\log v)}\left(u_{n}^{2}+\abs{\cos\theta}^{2}W^{2}b_{n}^{2}+2\abs{\cos\theta} Wu_nb_{n}\right)\right)
}
\eq{\label{defn:J_3}
\msJ_3
\coloneqq\sum_{i=1}^{n-1}u_{ii}^{2}\left(W-\frac{(1+\log v)}{(v\log v)}\abs{\cos\theta}^2b_{i}^{2}W^{2}\right),
}
as well as
\eq{\label{defn:J_4}
\msJ_4
&\coloneqq\sum_{k=1}^{n-1}u_{nk}b_{k}u_{n}u_{nn}\left(\underbrace{\frac{2\abs{\cos\theta}}{W^{2}}-\frac{(1+\log v)}{v\log v}\frac{2\abs{\cos\theta}}{W}-\frac{(1+\log v)}{(v\log v)}\frac{2\abs{\cos\theta}^2b_{n}}{u_{n}}}_{\coloneqq \msJ_{4,1}}\right)\\
&+\sum_{i=1}^{n-1}u_{ii}b_{i}u_{ni}\left(\underbrace{2\abs{\cos\theta}u_{n}-\frac{(1+\log v)}{(v\log v)}2\abs{\cos\theta}u_{n}W-\frac{(1+\log v)}{(v\log v)}2\abs{\cos\theta}^2b_{n}W^{2}}_{\coloneqq\msJ_{4,2}}\right)\\
&-\frac{(1+\log v)}{(v\log v)}\left(\sum_{k=1}^{n-1}b_{k}u_{kn}\right)^{2}\abs{\cos\theta}^2\\
=&\sum_{k=1}^{n-1}u_{nk}b_{k}u_{n}u_{nn}\msJ_{4,1}+\sum_{i=1}^{n-1}u_{ii}b_{i}u_{ni}\msJ_{4,2}-\frac{(1+\log v)}{(v\log v)}\left(\sum_{k=1}^{n-1}b_{k}u_{kn}\right)^{2}\abs{\cos\theta}^{2}.
}

To have a closer look at the term $\msJ_4$, we analyze the coefficients $\msJ_{4,1}$ and $\msJ_{4,2}$ in terms of sufficiently large $\abs{Du(z_0)}=u_n(z_0)$ (of course sufficiently large $v$ and $W$ at $z_0$) as follows:
\eq{\label{eq:J_4,1}
\msJ_{4,1}
\overset{\eqref{eq:v-z_0}}{=}&\frac{2\abs{\cos\theta}\left(W+\abs{\cos\theta}u_nb_{n}\right)u_{n}-2\abs{\cos\theta}Wu_{n}-2\abs{\cos\theta}^{2}W^{2}b_{n}}{W^{2}vu_{n}}+O\left(\frac{1}{u_{n}v\log v}\right)\\
=&\frac{-2\abs{\cos\theta}^{2}b_{n}}{W^{2}vu_{n}}+O\left(\frac{1}{u_{n}v\log v}\right),
}
where we have used the fact that $W^2=1+\abs{Du}^2=1+u_n^2$ at $z_0$.
Similarly, we have
\eq{\label{eq:J_4,2}
\msJ_{4,2}
&\overset{\eqref{eq:v-z_0}}{=}2\abs{\cos\theta}\frac{u_{n}\left(W+\abs{\cos\theta}u_{n}b_n\right)-u_{n}W-\abs{\cos\theta}W^{2}b_n}{v}+O\left(\frac{W}{\log v}\right)\\
&=-\frac{2\abs{\cos\theta}^{2}b_{n}}{v}+O\left(\frac{W}{\log v}\right).
}
Substituting \eqref{eq:J_4,1} and \eqref{eq:J_4,2} back into \eqref{defn:J_4}, we obtain
\eq{\label{eq:J_4}
\msJ_{4}
=&\sum_{k=1}^{n-1}u_{nk}b_{k}u_{n}u_{nn}\left(\frac{-2\abs{\cos\theta}^{2}b_n}{W^{2}vu_{n}}+O\left(\frac{1}{u_{n}v\log v}\right)\right)\\
&+\sum_{i=1}^{n-1}u_{ii}b_{i}u_{ni}\left(-\frac{2\abs{\cos\theta}^{2}b_{n}}{v}+O\left(\frac{W}{\log v}\right)\right)-\frac{(1+\log v)}{(v\log v)}\left(\sum_{k=1}^{n-1}b_{k}u_{kn}\right)^{2}\abs{\cos\theta}^{2},
}
and by Cauchy inequality we thus find
\eq{\label{ineq:J_4}
\msJ_4
\geq&-{C}\frac{\left(\sum_{k=1}^{n-1}u_{nk}b_{k}\right)^{2}}{v\log v}-{C}\frac{u_{nn}^{2}}{v\log v}
-{C}\frac{\sum_{i=1}^{n-1}u_{ii}^{2}W}{\log v}\\
&-{C}\frac{\sum_{i=1}^{n-1}u_{ni}^{2}W}{\log v}
-\frac{(1+\log v)}{(v\log v)}\left(\sum_{k=1}^{n-1}b_{k}u_{kn}\right)^{2}\abs{\cos\theta}^{2},
}
where we have used the fact that in the first bracket in \eqref{eq:J_4}, $O\left(\frac1{u_nv\log v}\right)$ is the dominating term when $\abs{Du(z_0)}$ is sufficiently large and a similar fact for
the second bracket in \eqref{eq:J_4}.
Here $ C$ may vary from line to line, but only depends $n,\theta$.

Now we go back to \eqref{eq:J expression} and use \eqref{defn:J_1}, \eqref{defn:J_2}, \eqref{defn:J_3} as well as \eqref{ineq:J_4} to obtain, for sufficiently large $W$, 
\eq{\label{eq:J}
\msJ
\geq&u_{nn}^{2}\left(\underbrace{\frac{2u_{n}^{2}}{W^{3}}+\frac{2\abs{\cos\theta}u_{n}b_{n}}{W^{2}}+\frac{1}{W^{3}}-\frac{(1+\log v)}{v\log v}\left(\frac{u_{n}}{W}+\abs{\cos\theta}b_{n}\right)^{2}-{C}\frac{1}{v\log v}}_{\coloneqq \mcC_{nn}}\right)\\
&+\sum_{i=1}^{n-1}u_{ni}^{2}\left(\underbrace{\frac{2u_{n}^{2}}{W}+2\abs{\cos\theta}u_nb_{n}+\frac{2}{W}-\frac{(1+\log v)}{v\log v}\left(u_{n}+\abs{\cos\theta}Wb_n\right)^{2}-{C}\frac{W}{\log v}}_{\coloneqq\mcC_{ni}}\right)\\
&+\sum_{i=1}^{n-1}u_{ii}^{2}\left(\underbrace{W-\frac{(1+\log v)}{(v\log v)}\abs{\cos\theta}^{2}W^{2}b_i^2-C(n)\frac{W}{\log v}}_{\coloneqq\mcC_{ii}}\right).
}

\

\noindent{\bf Step 2.3. We further estimate the coefficients appearing in \eqref{eq:J}.}

In terms of sufficiently large $W$, we have (recall \eqref{eq:v-z_0}, \eqref{defn:A})
\eq{\label{eq:AW-v}
W+\abs{\cos\theta}u_nb_n
=v
\approx u_n(1+\abs{\cos\theta}b_n)
\approx AW
=u_n+\abs{\cos\theta}Wb_n,
}
by virtue of which we obtain the refined estimates on the coefficients of $u_{nn}^2, u_{ni}^2,$ and $u_{ii}^2$ in \eqref{eq:J} as follows:
\eq{\label{eq:coefficients of unn^2}
\mcC_{nn}
=&\frac{2u_nAW}{W^3}+\frac1{W^3}-\frac{A^2W^2}{vW^2}-C\frac1{v\log v}\\
\geq&\underbrace{\frac{(u_{n}+\abs{\cos\theta}Wb_n)u_{n}}{W^{3}}-C\frac{1}{v\log v}}_{\coloneqq\mathbf{C}_n}.
}
Similarly,
the coefficients of $u_{ni}^{2}$ are estimated by
\eq{\label{eq:coefficients of uni^2}
\mcC_{ni}
=&\frac{2u_n}W(AW)+\frac2W-\frac1vA^2W^2-C\frac{W}{\log v}
=AW\left(\frac{2u_n}W-\frac{AW}v\right)-C\frac{W}{\log v}\\
\geq&AW\frac{u_n}W-C\frac{W}{\log v}
=\left(u_{n}+\abs{\cos\theta}Wb_{n}\right)\frac{u_{n}}{W}-C\frac{W}{\log v},
}
which is positive for sufficiently large $W$.
And the coefficients of $u_{ii}^{2}$ ($i=1,\ldots,n-1$) are estimated by
\eq{\label{eq:coefficients of uii^2}
\mcC_{ii}
\overset{\eqref{eq:v-z_0}}{=}&\underbrace{\frac{W(W+\abs{\cos\theta}u_nb_n)-\abs{\cos\theta}^{2}W^{2}b_i^{2}}{v}-C\frac{W}{\log v}}_{\coloneqq\mathbf{C}_i}.
}

\noindent{\bf Step 2.4. We finish this step by using the interior maximality \eqref{ineq:maximun-principle}.}

First note that by \eqref{eq:MSE}
we have
\eq{\label{eq:a^ij-varphi^ij}
a^{ij}\frac{\varphi_{ij}}\varphi
=\frac{\varphi'}{\varphi}{HW^3},
}
Then we apply Lemma \ref{Lem:cut-off}, in conjunction with \eqref{eq:J expression}, \eqref{eq:J}, \eqref{eq:coefficients of unn^2}, \eqref{eq:coefficients of uni^2}, \eqref{eq:coefficients of uii^2}, to obtain that
\eq{
0
\overset{\eqref{ineq:maximun-principle}}{\ge}&a^{ij}\left(\frac{\varphi_{ij}}{\varphi}-\frac{\varphi_{i}}{\varphi}\frac{\varphi_{j}}{\varphi}+\frac{\psi_{ij}}{\psi}-\frac{\psi_{i}}{\psi}\frac{\psi_{j}}{\psi}+\frac{v_{ij}}{v\log v}-\frac{(1+\log v)v_{i}v_{j}}{v^{2}\log v^{2}}\right)\\
\ge&{\frac{\varphi'}{\varphi}HW^3-u^2_n\left(\frac{\varphi'}{\varphi}\right)^{2}}-C\frac{1}{\psi r^{2}}-CW^{2}\frac{1}{\psi r^{2}}\\
+&\frac{\msJ_1+\msJ_2+\msJ_3+\msJ_4}{v\log v}+\frac{\frac{u_n}{W}{\mcT_n}+\abs{\cos\theta}\sum_{k=1}^{n}b_k{\msH_k}}{v\log v}\\
\geq&\left(\frac{\left(u_{n}+\abs{\cos\theta}Wb_n\right)u_{n}}{W^{3}}-C\frac{1}{v\log v}\right)\frac{u_{nn}^{2}}{v\log v}\label{eq:final expression}
+\sum_{i=1}^{n-1}\mathbf{C}_i
\frac{u_{ii}^{2}}{v\log v}\\
&{+\frac{\varphi'}{\varphi}HW^3-u^2_n\left(\frac{\varphi'}{\varphi}\right)^{2}}-C\frac{1}{\psi r^{2}}-CW^{2}\frac{1}{\psi r^{2}}+\frac{\frac{u_n}{W}{\mcT_n}+\abs{\cos\theta}\sum_{k=1}^{n}b_k{\msH_k}}{v\log v}.
}
By \eqref{eq:logG_i} and \eqref{eq:v_i} we have
\eq{
\frac{\varphi'}{\varphi}u_n+\frac{\psi_n}{\psi}
=-\frac{v_n}{v\log v}
=-\frac{\frac{u_n}Wu_{nn}+\abs{\cos\theta}u_{kn}b_k}{v\log v},
}
and hence
\eq{
&\frac{\varphi'}{\varphi}HW^3+\frac{\frac{u_n}{W}{\mcT_n}+\abs{\cos\theta}\sum_{k=1}^{n}b_k{\msH_k}}{v\log v}\\
=&\frac{\varphi'}{\varphi}HW^3
+\frac{u_nW^2H_n+\abs{\cos\theta}\sum_{k=1}^nb_kH_kW^3}{v\log v}+HWu_n\frac{\frac{u_n}Wu_{nn}+\abs{\cos\theta}\sum_{k=1}^nb_ku_{nk}}{v\log v}\\
=&\frac{\varphi'}{\varphi}HW^3
+\frac{u_nW^2H_n+\abs{\cos\theta}\sum_{k=1}^nb_kH_kW^3}{v\log v}-HWu_n\left(\frac{\varphi'}\varphi u_n+\frac{\psi_n}\psi\right)\\
=&\frac{\varphi'}{\varphi}HW-HWu_n\frac{\psi_n}{\psi}+\frac{u_nW^2H_n+\abs{\cos\theta}\sum_{k=1}^nb_kH_kW^3}{v\log v}\\
\geq&-C_H\frac{\varphi'}\varphi W-C_HW^2\frac1{\psi^\frac12r}-C_H\frac{W^2}{\log v},
}
where we have adopted the symbol $C_H$ to denote the constants depending only on $n,\theta,\mathbf{C}_H$, and  $C_H\equiv0$ if $H\equiv0$ (i.e., for minimal surface equation).
Note that by Cauchy-Schwarz inequality, we have $-C_H\frac{\varphi'}{\varphi}W\geq-C_H\left(\frac{\varphi'}{\varphi}\right)^2-C_HW^2$
and $-C_HW^2\frac1{\psi^\frac12r}\geq-C_HW^2-C_HW^2\frac1{\psi r^2}$, which give
\eq{\label{ineq:W^3-mcT_n-msH_k}
\frac{\varphi'}{\varphi}HW^3+\frac{\frac{u_n}{W}{\mcT_n}+\abs{\cos\theta}\sum_{k=1}^{n}b_k{\msH_k}}{v\log v}
\geq-C_H\left(\frac{\varphi'}\varphi\right)^2-C_HW^2-C_HW^2\frac1{\psi r^2}.
}

Finally we use \eqref{eq:lower bound og unn2}
to further estimate \eqref{eq:final expression} and get
\eq{\label{ineq:final-esti-step2}
0
\geq&\frac34\frac{A}{2A^2W}u_n^2v\log v\left(\frac{\varphi'}\varphi\right)^2+\sum_{i=1}^{n-1}\left(\mathbf{C}_i-\frac{C(\theta)}{Wv\log v}\right)u_{ii}^2\\
&{\underbrace{-C\left(\frac{\varphi'}\varphi\right)^2-u^2_n\left(\frac{\varphi'}{\varphi}\right)^{2}}_{\coloneqq I_{12}}-C_HW^2-CW^2\frac1{\psi r^2}}-C\frac{1}{\psi r^{2}}.
}
Until now, we have not used the angle assumption on $\theta$.
For the rest of {\bf Step 2}, we restrict ourselves to $\abs{\cos\theta}<\frac{\sqrt{3}}2$.

Since $\abs{\cos\theta}<\frac{\sqrt{3}}2$,
we have for sufficiently large $W$ and each $1\le i\le n-1$ the following estimate: 
\eq{\label{ineq:C_i}
\mathbf{C}_i
&\geq\frac{W(W+\abs{\cos\theta}u_nb_n)-\abs{\cos\theta}^2W^2\left(1-b_{n}^{2}\right)}{v}-C\frac{W}{\log v}\\
&=\frac{W^{2}(1-\abs{\cos\theta}^{2})+\abs{\cos\theta} Wu_nb_n+\abs{\cos\theta}^2W^{2}b_n^2}{v}-C\frac{W}{\log v}\\
&=\frac{\left(\abs{\cos\theta}Wb_n+\frac{u_{n}}{2}\right)^{2}-\frac{u_{n}^{2}}{4}+W^{2}(1-\abs{\cos\theta}^{2})}{v}-C\frac{W}{\log v}\\
&\geq\frac{W^{2}\left(\frac{3}{4}-\abs{\cos\theta}^{2}\right)}{2v}-C\frac{W}{\log v}.
}
Note that for sufficiently large $W$:
\begin{itemize}
    \item We have $\left(\mathbf{C}_i-\frac{C(\theta)}{Wv\log v}\right)>0$ thanks to \eqref{ineq:C_i};
    \item We have \eqref{eq:AW-v};
    \item $I_{12}$ can be absorbed into $\frac34\frac{A}{2A^2W}\left(\frac{\varphi'}\varphi\right)^2u_n^2v\log v$.
\end{itemize}
Further,
by definition of $\varphi$ we have
\eq{
\varphi(u)\in[\frac12,\frac32],
\quad
\varphi'(u)=\frac1{2M},
}
hence by \eqref{ineq:final-esti-step2}, we arrive at
\eq{
0\geq\frac1{25M^2}u_n^2\log v-C_HW^2
-CW^{2}\frac{1}{\psi r^{2}}.
}
Rearranging we obtain at $z_0$
\eq{
\psi\log v
\leq C_HM^2+C\left(\frac{M}{r}\right)^{2}
=C_1+C_2\frac{M^2}{r^2}.
}
In particular, if $H\equiv0$, then the above estimate reads (recall that $C_H$ are constants such that $C_H\equiv0$ if $H\equiv0$, see the discussion above \eqref{ineq:W^3-mcT_n-msH_k})
\eq{\label{esti:psi-logv-cos-theta<3/4}
\psi\log v
\leq C\left(\frac{M}{r}\right)^{2}=C\left(\frac{\sup_{E_r}\abs{u}+r}r\right)^2,
}
for some $C>0$ depends only on $n,\theta$.

Since $z_0$ is the maximum point of $G$, in virtue of \eqref{Basic properties of auxiliary functions}, these estimates, in conjunction with \eqref{esti:Thm3.2-bdry-maximum}, imply the required estimates \eqref{eq:Du-C_1-C_2-C_3-i}, and also \eqref{eq:Du-C_1-C_2-C_3-ii} for the case $\abs{\cos\theta}<\frac{\sqrt{3}}2$.

\

\noindent{\bf Step 3.
We show \eqref{eq:Du-C_1-C_2-C_3-ii} under the assumptions (on $n,\theta$) of ($ii$).
}

This step amounts to be an refinement of {\bf Step 2} and the main efforts are to deal with $$\left(\mathbf{C}_i-\frac{C(\theta)}{Wv\log v}\right)u_{ii}^2.$$ To do so,  by relabeling $1,\ldots,n-1$, we  may assume Wlog that $b_{1}^{2}\ge b_{2}^{2}\ge\cdots b_{n-1}^{2}$, where $\{b_i\}$ are coefficients appearing in \eqref{eq:v-z_0} and  satisfy $\sum_{i=1}^{n-1}b_i^2=1-b_n^2$.

We consider in the following  only the case $n\geq3$. The case $n=2$ is rather simple and we leave the details to the intrested reader.
Our analysis is based on the following observation:

\

{\bf Claim.}   {\it Only the coefficient of $u_{11}^{2}$ in \eqref{ineq:final-esti-step2}
could be negative, the other coefficients of $u_{ii}^2$ for $i\in\{2,\ldots,n-1\}$ must be positive, in terms of sufficiently large $W$.
In particular, if this is the case, then
for all $i\in\{2,\ldots,n-1\}$, the coefficients of $u_{ii}^2$ in \eqref{ineq:final-esti-step2} are positive with order at least $O\left(\frac{W^2}v\right)$.
}

\

To see this, we first observe that since $b_1^2\geq\ldots \geq b_{n-1}^2$, the quantities $\mathbf{C}_i$ defined in \eqref{eq:coefficients of uii^2} satisfy $\mathbf{C}_1\leq\ldots\leq\mathbf{C}_{n-1}$.
A direct computation gives: for any $i\in\{2,\ldots,n-1\}$,
\eq{
&\left(\mathbf{C}_1-\frac{C(\theta)}{Wv\log v}\right)
+\left(\mathbf{C}_i-\frac{C(\theta)}{Wv\log v}\right)\\
\overset{\eqref{eq:coefficients of uii^2}}{=}&\frac{2W(W+\abs{\cos\theta}u_nb_n)-\abs{\cos\theta}^{2}W^{2}(b_1^{2}+b_i^2)}{v}-C\frac{W}{\log v}
-C\frac1{Wv\log v}\\
\geq&\underbrace{\frac{W^2}v\left((1-\abs{\cos\theta}^2)+\left(\abs{\cos\theta}b_n+1\right)^2\right)}_{\coloneqq {\bf I}_2}+\frac{W^2}v\left(\frac{2\abs{\cos\theta}b_n(u_n-W)}{W}\right)\underbrace{-\frac{CW}{\log v}
-\frac{C}{Wv\log v}}_{\coloneqq{\bf I}_3}\\
\geq&\frac{W^2}v\left(\frac12(1-\abs{\cos\theta}^2)+\left(\abs{\cos\theta}b_n+1\right)^2\right),
}
where we have used the fact that $b_1^2+b_i^2\leq1-b_n^2$ in the first inequality, and  in the last inequality we absorbed ${\bf I}_3$ into ${\bf I}_2$, since the negative terms $-C\frac{W}{\log v}-C\frac1{Wv\log v}$, together with the fact that the term $\frac{W^2}v\frac{2\abs{\cos\theta}b_n(u_n-W)}W=\frac{W^2}vO\left(\frac{u_n}W-1\right)$ can be controlled by the first term of ${\bf I}_2$ in terms of sufficiently large $\abs{Du(z_0)}$, thanks to the fact that $(1-\abs{\cos\theta})^2>0$. 
In particular, if $\left(\mathbf{C}_1-\frac{C(\theta)}{Wv\log v}\right)<0$, then {\bf Claim} follows immediately from the above estimate.

Thus, the only scenario that we need to be worried about is when $\mathbf{C}_1-\frac{C(\theta)}{Wv\log v}<0$, otherwise the argument in {\bf Step 2} applies and  the proof is  completed.
{
Moreover, back to the expression of $\mathbf{C_1}$ (recall \eqref{eq:coefficients of uii^2}) we see, in terms of sufficiently large $\abs{Du(z_0)}$, we could absorb the term $-\frac{C(\theta)}{Wv\log v}$ into $-C\frac{W}{\log v}$.
So instead of assuming $\mathbf{C}_1-\frac{C(\theta)}{Wv\log v}<0$, let us assume $\mathbf{C}_1\leq0$ in all follows.
}

To proceed, we rewrite \eqref{eq:MSE-z_0} to find (recall that we have assumed $H\equiv0$)
\eq{
u_{11}
=-\sum_{i=2}^{n-1}u_{ii}-\frac{u_{nn}}{W^{2}},
}
and then we have
\eq{\label{u11}
u_{11}^2
=&\left(\sum_{i=2}^{n-1}u_{ii}\right)^2+\left(\frac{u_{nn}}{W^{2}}\right)^2+2\left(\sum_{i=2}^{n-1}u_{ii}\right)\left(\frac{u_{nn}}{W^{2}}\right)\\
\leq&(n-2)\sum_{i=2}^{n-1}u_{ii}^2+\left(\frac{u_{nn}}{W^{2}}\right)^2+\frac{\varepsilon_0}{2(n-2)}\left(\sum_{i=2}^{n-1}u_{ii}\right)^2+\frac{2(n-2)}{\varepsilon_0}\left(\frac{u_{nn}}{W^{2}}\right)^2\\
\leq&\left(n-2+\frac{\varepsilon_0}{2}\right)\sum_{i=2}^{n-1}u_{ii}^2+\frac{2(n-2)+\varepsilon_0}{\varepsilon_0}\left(\frac{u_{nn}}{W^{2}}\right)^2,
}
where $\varepsilon_0$ is determined later.
Now we use \eqref{eq:u_nn}, \eqref{defn:A}, \eqref{ineq:ep_3-varphi}, and \eqref{eq:AW-v} to further estimate that
\eq{
\frac{\abs{u_{nn}}}{W^2}
\leq\frac{C(\theta)}{W^2}\sum_{i=1}^{n-1}\abs{u_{ii}}+2\frac{v\log v}{AW^2}\frac{\varphi'u_n}{\varphi}
\approx O\left(\frac1{W^2}\right)\sum_{i=1}^{n-1}\abs{u_{ii}}+2\log v\frac{\varphi'}{\varphi},
}
and hence
\eq{
\left(\frac{u_{nn}}{W^2}\right)^2
\le O\left(\frac1{W^4}\right)u_{11}^2+ O\left(\frac1{W^4}\right)\sum_{i=2}^{n-1}u_{ii}^2+8(\log v)^{2}\left(\frac{\varphi'}{\varphi}\right)^{2}.
}
Together with \eqref{u11}, we obtain
\eq{\label{ineq:u_11^2}
u_{11}^2
\leq\left(n-2+\varepsilon_{0}\right)\sum_{i=2}^{n-1}u_{ii}^{2}+C(\varepsilon_{0})(\log v)^{2}\left(\frac{\varphi'}{\varphi}\right)^{2}.
}

Recalling  the definition of $\mathbf{C}_{i}$ in {\bf Step 2.3}, by \eqref{ineq:u_11^2} and {$\mathbf{C}_1\leq0$},  we obtain 
\eq{\label{eq:rerangment of uii^2}
&\sum_{i=1}^{n-1}(\mathbf{C}_i-\frac{C(\theta)}{Wv\log v})\frac{u_{ii}^{2}}{v\log v}\\
\ge&\sum_{i=2}^{n-1}\bigg(\frac{W(W+\abs{\cos\theta}u_nb_n)-\abs{\cos\theta}^{2}W^{2}b_i^{2}}{v}\\
+&\left(n-2+\varepsilon_{0}\right)\left(\frac{W(W+\abs{\cos\theta}u_nb_n)-\abs{\cos\theta}^{2}W^{2}b_1^{2}}{v}\right)-C\frac{W}{\log v}\bigg){\frac{u_{ii}^{2}}{v\log v}}
-C(\varepsilon_{0})\log v\left(\frac{\varphi'}{\varphi}\right)^{2}\\
\ge&{\frac{W^2}{v^2\log v}\sum_{i=2}^{n-1}\mcB_{i}}u_{ii}^{2}-C{\frac{W}{v(\log v)^2}}\sum_{i=2}^{n-1}u_{ii}^{2}-C(\varepsilon_{0})\log v\left(\frac{\varphi'}{\varphi}\right)^{2}
}
where 
for $2\le i\le n-1$,
\eq{
\mcB_{i}
\coloneqq 1+\abs{\cos\theta}b_{n}-\abs{\cos\theta}^2b_{i}^{2}+\left(n-2+\varepsilon_{0}\right)\left(1+\abs{\cos\theta}b_{n}-\abs{\cos\theta}^{2}b_1^{2}\right).
}
Recall that $\sum_{i=1}^nb_i^2=1$,
elementary computations then give
\eq{\label{defn:msB}
\mcB_{i}
&\geq\left(n-1+\varepsilon_{0}\right)\left(1+\abs{\cos\theta}b_{n}\right)-\abs{\cos\theta}^2(1-b_{n}^{2})-(n-3+\varepsilon_{0})\abs{\cos\theta}^{2}b_1^{2}\\
&\ge(n-1+\varepsilon_{0})(1+\abs{\cos\theta}b_{n})-\abs{\cos\theta}^2(1-b_{n}^{2})-(n-3+\varepsilon_{0})\abs{\cos\theta}^2(1-b_{n}^{2})\\
&=(n-1+\varepsilon_{0})-(n-2+\varepsilon_{0})\abs{\cos\theta}^{2}+(n-1+\varepsilon_{0})\abs{\cos\theta}b_{n}+(n-2+\varepsilon_{0})\abs{\cos\theta}^{2}b_n^{2}\\
&=\left(n-2+\varepsilon_{0}\right)\left(\abs{\cos\theta} b_{n}+\frac{n-1+\varepsilon_{0}}{2(n-2+\varepsilon_{0})}\right)^{2}-\frac{(n-1+\varepsilon_{0})^{2}}{4(n-2+\varepsilon_{0})}\\
&\quad+(n-1+\varepsilon_{0})-(n-2+\varepsilon_{0})\abs{\cos\theta}^{2}\\
&\geq-\frac{(n-1+\varepsilon_{0})^{2}}{4(n-2+\varepsilon_{0})}+(n-1+\varepsilon_{0})-(n-2+\varepsilon_{0})\abs{\cos\theta}^{2}
\eqqcolon\msB.
}
The ideal situation is that $\msB$ is a positive number, which is equivalent to the inequality
\eq{\label{eq:further condition}
\frac{n-1+\varepsilon_{0}}{n-2+\varepsilon_{0}}\left(1-\frac{n-1+\varepsilon_{0}}{4(n-2+\varepsilon_{0})}\right)
>\abs{\cos\theta}^{2}.
}
To make sure that we could always find some $\varepsilon_0=\varepsilon_0(n,\theta)$ to fulfill \eqref{eq:further condition}, notice that when $\varepsilon_0=0$ the LHS of the above inequality reads $\frac{(3n-7)(n-1)}{4(n-2)^{2}}$, and hence by monotonicity, condition \eqref{condi:msU} ensures the existence of $\varepsilon_0$.
Note that for $n=3$, \eqref{condi:msU} is trivially satisfied by all $\theta\in(0,\pi)$.
We emphasize that
this is the only place that we need to assume \eqref{condi:msU}, and let us fix one such $\varepsilon_0=\varepsilon_0(n,\theta)<1$ fulfilling \eqref{eq:further condition} in all follows.

Back to \eqref{eq:rerangment of uii^2}, we thus obtain for sufficiently
large $W$:
\eq{\label{ineq:msB}
\sum_{i=1}^{n-1}(\mathbf{C}_i-\frac{C(\theta)}{Wv\log v})\frac{u_{ii}^{2}}{v\log v}
\geq\msB\frac{W^2}{v\log v}\sum_{i=2}^{n-1}u_{ii}^{2}-C\frac{W}{v(\log v)^2}\sum_{i=2}^{n-1}u_{ii}^{2}-C(\varepsilon_{0})\log v\left(\frac{\varphi'}{\varphi}\right)^{2}.
}
Finally, by \eqref{ineq:final-esti-step2}, 
\eqref{ineq:msB}, 
\eq{
0\ge&\frac34\frac{A}{2A^{2}W}u_{n}^{2}v\log v\left(\frac{\varphi'}{\varphi}\right)^{2}-C\log v\left(\frac{\varphi'}{\varphi}\right)^{2}
-u_n^{2}\left(\frac{\varphi'}{\varphi}\right)^{2}-CW^{2}\frac{1}{\psi r^{2}}\\
&+\sum_{i=2}^{n-1}\msB\frac{W^2}{v\log v}u_{ii}^{2}-C\frac{W}{v(\log v)^2}\sum_{i=2}^{n-1}u_{ii}^{2}
.
}
In the above estimate,
since the coefficients of the terms involving $u_{ii}^2$ are given by
\eq{
\msB\frac{W^2}{v\log v}
-C\frac{W}{v(\log v)^2},
}
and $\msB$ is a positive number,   the coefficients of $u_{ii}^2$ are positive provided that $W$ is sufficiently large. This  implies 
(recall \eqref{eq:AW-v})
\eq{\label{ineq:final-esti-step3}
0
\ge&\frac14u_{n}^{2}\log v\left(\frac{\varphi'}{\varphi}\right)^{2}-C\left(n,\theta\right)\log v\left(\frac{\varphi'}{\varphi}\right)^{2}
-u_n^{2}\left(\frac{\varphi'}{\varphi}\right)^{2}
-CW^{2}\frac{1}{\psi r^{2}}\\
\geq&\frac18u_{n}^{2}\log v\left(\frac{\varphi'}{\varphi}\right)^{2}-CW^{2}\frac{1}{\psi r^{2}}.
}
Concluding as the end in  {\bf Step 2.4}, we thus obtain \eqref{eq:Du-C_1-C_2-C_3-ii} under the assumptions of ($ii$).

Finally, if $u$ has linear growth, then $\frac{M}{r}=\frac{\sup_{E_r}\abs{u(x)}+r}r\leq C(\theta,C_0)$.
Letting $r\ra\infty$ in the estimate \eqref{eq:Du-C_1-C_2-C_3-ii}, and recall that $\lim_{r\ra\infty}E_{\theta,r}=\mathbb{R}^n_+$, we thus obtain \eqref{conclu:Du-bdd}, which completes the proof.
\end{proof}


In view of {\bf Step  3} in the above proof, we see that the angle restriction is mainly due to the negativity of the coefficient of the term $u_{11}^2$ appearing in  \eqref{ineq:maximun-principle}, where the subscript $1$ refers to $\tilde e_1$, the direction that majorly contributes to $\sum_{i=1}^n\left<-e_1,\tilde e_i\right>^2=1-\left<-e_1,\tilde e_{n+1}\right>^2$ in the rotated coordinates.
In principle $\left<-e_1,\tilde e_1\right>$ could be as close as to $1$, which makes the angle restriction look essential.

A way to overcome this, as we are going to see, is to impose an extra condition that $u$ is one-sided bounded by a linear function on the half-space.

\subsection{Global gradient estimate for minimal surface equation 
}\label{Sec-3-2}

\begin{theorem}\label{Thm:gradient-esti-u<=0}
Let 
$\theta\in(0,\pi)$ and  $u$ be a $C^2$-solution of the mean curvature equation \eqref{eq:MSE}, such that its graph $\S$ is a capillary minimal graph in the sense of Definition \ref{Defn:capillary-graph}.
Assume that 
$u$ has linear growth on $\mathbb{R}^n_+$, namely, $\abs{u(x)}\leq C_{0}(1+\abs{x})$ for some constant $C_0>0$.

There exists a positive constant $\widehat\Lambda=\widehat\Lambda(n,\theta,C_0)$ with the following property:
If $u$ is bounded from above by a linear function $L$ on $\mathbb{R}^n_+$, with $\abs{DL}\leq\frac1{36}\frac {\abs{\cos\theta}(1-\sin\theta)}{(1+\frac{|\cos\theta |}{\sin\theta})}\eqqcolon C_\theta$,
then
\eq{\label{conclu:gradient-estimate-bounded-above-L}
\sup_{\overline{\mathbb{R}^n_+}}\abs{Du}
\leq\widehat\Lambda.
}
\end{theorem}

\begin{proof}
Recalling Remark \ref{Rem:theta>pi/2},
in the following we only consider those $\theta\in(0,\pi)\setminus\{\frac\pi2\}$.

\noindent{\bf Step 1. We construct modified cut-off functions and set things up.
}

In contrast to Lemma \ref{Lem:cut-off}, we consider the following modified cut-off function:
For any large $r>0$,
define
\eq{\label{defn:bar-psi}
Q^\ast(x)
=\left(\underbrace{1-\frac{(x_{1}-\abs{\cos\theta}r)^{2}+\sin^{2}\theta\abs{x'}^{2}}{r^{2}}}_{=Q(x)}+\frac{u(x)-L(x)}{2N_\ast r}\right)
}
and 
\eq{\label{defn:psi^ast}
\psi^\ast
=\left(Q^\ast\right)^2
}
where we can choose $N_\ast=\frac{1}{36}$, determined at the end of the proof.

Then we fix a ``inner'' region which enlarges as $r$ increases and converges to $\mathbb{R}^n_+$ as $r\ra\infty$, on which $\psi^\ast$ has an absolute lower bound independent on $r$, playing the same role as $E_{\theta,r}$ in the proof of Theorem \ref{Thm:gradient-estimate}.
In fact, we assert that there exists a $W_{r_\ast}\subset E_{\theta, r}$ such that $\lim_{r\ra\infty}W_{r_\ast}=\mathbb{R}^n_+$ and 
\eq{\label{ineq:lower-bdd-psi^ast}
\inf_{W_{r_\ast}}\psi^\ast
>c^*>0
}
where the positive constant $c^*$ depends on $\theta,n, C_0$.
Here for example, a possible choice is to take $W_{r_\ast}=B_{\alpha r}^+
\coloneqq\left\{(x_1,x'):x_1>0, \abs{x_1}^2+\abs{x'}^2< \alpha^2r^2\right\}$ for some suitably chosen $\alpha=\alpha(n,\theta,C_0)>0$.
Once the required ``inner'' region $W_{r_\ast}$ is fixed, we can then take the ``outer'' region simply to be the $0$-super level-set of $Q^*$, i.e.,
\eq{\label{defn:W_r}
W_r\coloneqq\{x:x_1>0, Q^\ast>0\},
}
which strictly contains $W_{r_\ast}$, and plays the same role as $E_r$ in the proof of Theorem \ref{Thm:gradient-estimate}.

Clearly $W_r$ is bounded, relatively open in $\mathbb{R}^n_+$, and in fact we have $W_r\subset E_r$ (recall \eqref{Er}) thanks to the fact that $u-L\leq0$.

We then conduct some useful computations needed in the next step.
By \eqref{defn:psi^ast}
we have
\eq{\label{eq:p_1-psi}
\p_1\psi^\ast
&=2(\psi^\ast)^{\frac{1}{2}}\left(-\frac{2}{r^{2}}(x_{1}-\abs{\cos\theta} r)+\frac{u_{1}-L_1}{2N_\ast r}\right),
}
and hence on $\p\mathbb{R}^n_+$, 
\eq{
\p_1\psi^\ast
=\frac{4}{r}(\psi^\ast)^{\frac{1}{2}}\abs{\cos\theta}+\frac{(\psi^\ast)^{\frac{1}{2}}(u_{1}-L_1)}{N_\ast r}.
}
And for for $i\in\{2,\ldots,n\}$, we have
\eq{\label{eq:p_i-psi}
\p_i\psi^\ast
&=2(\psi^\ast)^{\frac12}\left(-\frac{2\sin^{2}\theta}{r^{2}}x_{i}+\frac{u_i-L_i}{2N_\ast r}\right).
}

We finish this step by building up the frameworks for the following discussions.
As in the proof of Theorem \ref{Thm:gradient-estimate} we consider the function 
\eq{
G^\ast(x)
=\varphi(u(x))\psi^\ast(x)\log v(x),
}
where $\varphi(s)=\frac{s}{2M_\ast}+1$ with $M_\ast=\sup_{W_r}\abs{u}+r$.
Assume that
\eq{
\max_{\overline{W_r}}G^\ast
=G^\ast(z_0).
}
Our goal is to show that $\sup_{\overline{W_{r_*}}}\abs{Du}$ is bounded by some constant independent of $r$.
By \eqref{ineq:lower-bdd-psi^ast}, we have
\eq{\label{ineq:lower-bd-varphi-psi-ast}
\frac12<\varphi(u)<\frac32,\quad
c^*\leq\psi^\ast\leq1,\quad\text{on }\overline{W_{r_\ast}}\subset\overline{W_r},
}
so we  assume that $G^\ast(z_0)$ is positive and sufficiently large, otherwise there is nothing to prove.
In this case, $\sup_{\overline{W_{r_*}}}\abs{Du}$, $\sup_{\overline{W_{r_*}}}W$, and $\sup_{\overline{W_{r_*}}}v$ are sufficiently large. 
By construction of $W_r$, we see $z_0\notin\p_{rel}W_r$.

The step is thus finished. We point out that,
in the following we shall refer to the computations carried out in the proof of Theorem \ref{Thm:gradient-estimate} from time to time, and if the cut-off function is involved, readers should replace automatically $\psi$ therein by the modified cut-off functions $\psi^\ast$.

\noindent{\bf Step 2. We deal with the case that $z_0\in\p W_r\setminus\p_{rel}W_r$.
}

First, since $\p_{rel}W_r$ is the $0$-level set of $\psi^\ast$, and that $u-L\leq0$ on $\mathbb{R}^n_+$ by assumption,
we see that (recall Lemma \ref{Lem:cut-off})
$\left(\p W_r\setminus\p_{rel}W_r\right)
\subset\left(\p E_r\setminus\p_{rel}E_r\right)$, from which we infer that
\eq{\label{esti:sin-theta-x'<r}
\abs{x'}<r, \quad\forall x\in\p W_r\setminus\p_{rel}W_r.
}

As in \eqref{eq:a^i1-psi_i},  by \eqref{eq:p_1-psi}, \eqref{eq:p_i-psi}, and \eqref{eq:bar-Du-u_1-bdry}, we compute
\eq{\label{eq:a^i1-psi_i^ast}
&a^{i1}\frac{\psi^\ast_{i}}{\psi^\ast}
=\frac{W^2}{\psi^\ast}\left(\psi^\ast_1-\frac{\sum_{i=1}^nu_iu_1\psi^\ast_i}{W^2}\right)\\
&{=}\frac{W^2}{\psi^\ast}\left(\left(\frac{4}{r}\abs{\cos\theta}+\frac{u_1-L_1}{N_\ast r}\right)(\psi^\ast)^{\frac{1}{2}} \frac{1+\abs{\bar Du}^2}{W^2}+\frac{\sum_{i=2}^nu_iu_1\left(\frac{4\sin^2\theta x_i}{r^2}-\frac{u_i-L_i}{N_\ast r}\right)(\psi^\ast)^\frac12}{W^2}\right)\\
&=\frac{1}{r(\psi^\ast)^\frac12}\left(4\abs{\cos\theta}(1+\abs{\bar Du}^2)+4\sin^2\theta\sum_{i=2}^nu_iu_1\frac{x_i}{r}+\frac{u_1}{N_\ast}-\frac{L_1(1+\abs{\bar Du}^2)-u_1\sum_{i=2}^nu_iL_i}{N_\ast}\right)\\
&\geq\frac{1}{r(\psi^\ast)^\frac12}\left(4\abs{\cos\theta}(1+\abs{\bar Du}^2)+4\sin^2\theta\sum_{i=2}^nu_iu_1\frac{x_i}{r}+\frac{u_1}{N_\ast}-(1+\frac{|\cos\theta |}{\sin\theta})\frac{\abs{DL}(1+\abs{\bar Du}^2)}{N_\ast}\right).
}
By \eqref{esti:sin-theta-x'<r}, for the same reason as in \eqref{eq:a^i1-psi_i}, we conclude that
\eq{
\left(4\abs{\cos\theta}(1+\abs{\bar Du}^2)+4\sin^2\theta\sum_{i=2}^nu_iu_1\frac{x_i}{r}\right)
\geq4\abs{\cos\theta}(1-\sin\theta)\sqrt{(1+\abs{\bar Du}^2)}\abs{\bar D u}.
}
Since (recall that $N_\ast=\frac1{36}$) by assumption $\abs{DL}\leq C_+
=N_* \frac {\abs{\cos\theta}(1-\sin\theta)}{(1+\frac{|\cos\theta |}{\sin\theta})}$, we thus obtain
\eq{
a^{i1}\frac{\psi^\ast_{i}}{\psi^\ast}
\geq\frac2{r(\psi^\ast)^\frac12}\abs{\cos\theta}(1-\sin\theta)\sqrt{(1+\abs{\bar Du}^2)}\abs{\bar D u}
>0.
}
Hence at  $z_0$, we find (recall \eqref{eq:na-v-mu=0}, we have $a^{ij}v_{i}(x_{1})_{j}=0$)
\eq{\label{ineq:max-bdry-u<=0-hopf-lemma}
0&\ge a^{ij}(\log G^\ast)_{i}(x_{1})_{j}
=a^{i1}\left(\frac{\varphi'u_{i}}{\varphi}+\frac{v_{i}}{v\log v}+\frac{\psi^\ast_i}{\psi^\ast}\right)
=a^{i1}\frac{\psi^\ast_i}{\psi^\ast}+a^{i1}\frac{\varphi'}{\varphi}u_{i}\\
\overset{\eqref{eq:a^i1-u_i}}{\ge}&\frac2{r(\psi^\ast)^\frac12}\abs{\cos\theta}(1-\sin\theta)\sqrt{(1+\abs{\bar Du}^2)}\abs{\bar D u}-\frac{1}{2M_\ast\varphi}(\abs{\cos\theta} W).
}
Recalling \eqref{eq:bar-Du-u_1-bdry}, we get $\abs{Du(z_0)}\leq C(\theta)\frac{r(\psi^\ast)^\frac12}{M_\ast}\leq C(\theta)\frac{r}{r+\sup_{W_r}\abs{u}}\leq C(\theta)$.
Since $z_0$ is the maximum point of $G^\ast$, we thus find
\eq{\label{esti:Thm3.3-bdry-maximum}
C(n,\theta)\log v(x)
\overset{\eqref{ineq:lower-bd-varphi-psi-ast}}{\leq}\varphi(u(x))\psi^\ast(x)\log v(x)
=G^\ast(x)
\leq G^\ast(z_0)
\leq C(n,\theta),\quad\forall x\in W_{r_\ast}.
}
The step is thus completed.
Next we study the case that $z_0\in W_r$.

\noindent{\bf Step 3.
We carry out necessary estimates to exploit $0\geq a^{ij}(\log G^\ast)_{ij}$ at $z_0\in W_r$.
}

As in the proof of Theorem \ref{Thm:gradient-estimate}, we assume that $\abs{Du(z_0)}=u_{n}(z_0)$, and $\{u_{ij}(z_0)\}_{1\leq i,j\leq n-1}$ is a diagonal matrix. Also $v
=W+\abs{\cos\theta}\sum_{k=1}^nu_kb_k.$
Then we follow the computations in {\bf Step 2} of the proof of Theorem \ref{Thm:gradient-estimate}.

At $z_0$, by $(\log G^\ast)_{i}=0$ we get (see \eqref{eq:logG_i}, \eqref{eq:v_i})
\eq{\label{newauxiliaryfirstderivative}
v_i=\frac{u_{n}u_{ni}}{W}+\abs{\cos\theta} u_{ki}b_{k}
=-v\log v\left(\frac{\varphi_{i}}{\varphi}+\frac{\psi^\ast_{i}}{\psi^\ast}\right).
}

\noindent{\bf Step 3.1. We bound $u_{nn}^2$ from below as in \eqref{eq:lower bound og unn2}.
}

We put $\mfQ(x)\coloneqq Q(x)-\frac{L(x)}{2N_\ast r}$ for simplicity 
(recall \eqref{defn:bar-psi}).
Since $W_r\subset E_r$,
the function $\mfQ$ satisfies 
\eq{\label{ineq:DQ}
\abs{D\mfQ(x)}
\leq\abs{DQ}+\frac{\abs{DL}}{2N_\ast r}
\leq \frac{2}{r}+\frac{C_+}{2N_\ast r},\quad
{\abs{D^2\mfQ(x)}
=\abs{D^2Q(x)}
\leq C(n,\theta)\frac{1}{r^2},}\quad
\forall x\in W_r.
}
Thus, at $z_0$ we have 
\eq{
\psi^\ast_n
=2(\psi^\ast)^\frac12\left(\mfQ_{n}+\frac{u_{n}}{2N_\ast r}\right),
}
and
\eq{\label{eq:psi_i^ast}
\psi^\ast_{i}
=2(\psi^\ast)^\frac12\mfQ_{i},\quad i=1,\ldots,n-1.
}
By virtue of \eqref{ineq:DQ}, 
provided that $u_n(z_0)=\abs{Du}(z_0)$ is sufficiently large, we have at $z_0$
\eq{
\frac{u_{n}}{N_\ast r}
>\mfQ_{n}+\frac{u_{n}}{2N_\ast r}
>\frac{u_{n}}{4N_\ast r},
}
from which we infer (recall that $\varphi'=\frac1{2M_\ast}>0$)
\eq{\label{eq:varphi_n-psi_n}
\left(\frac{\varphi_{n}}{\varphi}+\frac{\psi^\ast_{n}}{\psi^\ast}\right)^{2}
=\left(\frac{\varphi'}{\varphi}u_{n}+\frac{2}{(\psi^\ast)^\frac12}\left(\mfQ_{n}+\frac{u_{n}}{2N_\ast r}\right)\right)^{2}
\ge\left(\frac{\varphi'}{\varphi}+\frac{1}{2N_\ast r(\psi^\ast)^\frac12}\right)^{2}u_{n}^{2}.
}
As in \eqref{eq:u_nn} we have
\eq{\label{unn}
u_{nn}
&=\frac{\cos^2\theta}{A^{2}}\sum_{i=1}^{n-1}b_{i}^{2}u_{ii}-\frac{v\log v}{A}\left(\frac{\varphi_{n}}{\varphi}+\frac{\psi^\ast_{n}}{\psi^\ast}\right)+\frac{\abs{\cos\theta}v\log v\sum_{k=1}^{n-1}b_{k}\psi^\ast_{k}}{A^{2}\psi^\ast},
}
where $A=\frac{u_n}W+\abs{\cos\theta}b_n$ satisfies the estimate \eqref{defn:A} when $G^*(z_0)$ is sufficiently large.
We assume that
there holds
\eq{\label{extraterm}
\frac18\left(\frac{\varphi_n}\varphi+\frac{\psi_n^\ast}{\psi^\ast}\right)
\geq\frac{\abs{\cos\theta}\sum_{k=1}^{n-1}\abs{b_k\psi_k^\ast}}{A\psi^\ast},
}
otherwise the proof is finished for the same reason as \eqref{ineq:sup-v-ep_3}.
Taking \eqref{eq:varphi_n-psi_n}, \eqref{unn}, and \eqref{extraterm} into account, we thus arrive at (compared to \eqref{eq:lower bound og unn2})
\eq{\label{eq:lower bound og unn2-1}
u_{nn}^{2}
&\ge\frac{3}{4A^{2}}u_{n}^{2}(v\log v)^{2}\left(\frac{\varphi'}{\varphi}+\frac{1}{2N_\ast r(\psi^\ast)^{\frac12}}\right)^{2}-C(\theta)\sum_{i=1}^{n-1}u_{ii}^{2}.
}

\noindent{\bf Step 3.2. We estimate the last two terms appearing in \eqref{ineq:maximun-principle}.}

Observe that from \eqref{eq:a^ij-u_ijk} to \eqref{eq:coefficients of uii^2}, the computations concern only the function $u$ and its derivatives, and have nothing to do with the cut-off function $\psi$ therein, which means these computations are still valid in this case.
Let us set
\eq{
\msP
\coloneqq&a^{ij}\left(\frac{v_{ij}}{v\log v}-\frac{(1+\log v)v_{i}v_{j}}{v^{2}\log v^{2}}\right).
}
Then by \eqref{eq:J expression} and \eqref{eq:J}, with same notations $\mcC_{nn},\mcC_{ni}, \mcC_{ii}$ as in the proof of Theorem \ref{Thm:gradient-estimate},  we obtain (recall that $H\equiv0$)
\eq{\label{ineq:P-1}
\msP
\geq\underbrace{\frac{\mcC_{nn}u_{nn}^{2}+\sum_{i=1}^{n-1}\mcC_{ni}u_{ni}^{2}+\sum_{i=1}^{n-1}\mcC_{ii}u_{ii}^{2}}{v\log v}}_{\coloneqq\msP_0}.
}

Upon relabeling the index of $\{1,\ldots,n-1\}$, we assume that $b_1^2\geq b_2^2\geq\ldots\geq b_{n-1}^2$, where $b_i$ are coefficients appearing in \eqref{eq:v-z_0}, and satisfy $\sum_{i=1}^{n-1}b_i^2=1-b_n^2$.

Now we break $\msP_0$ into two terms (recall that $u_1(z_0)=0$):
\eq{
\msP_0
&=\msP_0+\frac{a^{11}v_{1}^{2}(1+\log v)}{(v\log v)^{2}}-\frac{a^{11}v_{1}^{2}(1+\log v)}{(v\log v)^{2}}\\
&\overset{\eqref{newauxiliaryfirstderivative}}{=}\underbrace{\msP_0+\frac{a^{11}v_{1}^{2}(1+\log v)}{(v\log v)^{2}}}_{\coloneqq\msP_1}-W^{2}(1+\log v)\left(\frac{\psi^\ast_{1}}{\psi^\ast}\right)^{2}.
}

The term $\frac{a^{11}v_{1}^{2}(1+\log v)}{v\log v}$ in $ \msP_1$ can be simply estimated by virtue of \eqref{eq:a^ij-z_0} and \eqref{eq:v_i} as follows:
\eq{
\frac{a^{11}v_{1}^{2}(1+\log v)}{(v\log v)^{2}}
=&\frac{W^{2}(1+\log v)}{(v\log v)^{2}}\left(\abs{\cos\theta} u_{11}b_{1}+Au_{n1}\right)^{2}\\
\ge&\frac{W^{2}(1+\log v)}{(v\log v)^{2}}\left(\frac{1}{2}\abs{\cos\theta}^{2}u_{11}^{2}b_{1}^{2}-A^{2}u_{n1}^{2}\right).
}
This in turn gives
\eq{\label{ineq:P_1-1}
\msP_{1}
\overset{\eqref{ineq:P-1}}{\ge}&\frac{\mcC_{nn}u_{nn}^{2}+\sum_{i=1}^{n-1}\mcC_{ni}u_{ni}^{2}+\sum_{i=1}^{n-1}\mcC_{ii}u_{ii}^{2}}{v\log v}+\frac{W^{2}(1+\log v)}{(v\log v)^{2}}\left(\frac{1}{2}\abs{\cos\theta}^{2}u_{11}^{2}b_{1}^{2}-A^{2}u_{n1}^{2}\right)\\
=&\frac{1}{v\log v}\bigg(\mcC_{nn}u_{nn}^{2}+\sum_{i=2}^{n-1}\mcC_{ni}u_{ni}^{2}+\sum_{i=2}^{n-1}\mcC_{ii}u_{ii}^{2}\\
&+\left(\mcC_{n1}-A^{2}\frac{W^{2}(1+\log v)}{v\log v}\right)u_{n1}^{2}+\left(\mcC_{11}+\frac{1}{2}\abs{\cos\theta}^{2}b_{1}^{2}\frac{W^{2}(1+\log v)}{v\log v}\right)u_{11}^{2}\bigg).
}

\noindent{\bf Step 3.3. We further estimate the coefficients appearing in \eqref{ineq:P_1-1}.}

By virtue of the {\bf Claim} shown in {\bf Step 3} of the proof of Theorem \ref{Thm:gradient-estimate}, we only need to consider the scenario that $\mcC_{11}\leq0$ and $\mcC_{ii}
\geq O\left(\frac{W^2}v\right)>0$ for $i=2,\ldots,n-1$, otherwise the proof is finished for a similar reason as the discussion subsequent to the {\bf Claim}.

Now we take a closer look at \eqref{ineq:P_1-1}.
Recall the definition of $\mcC_{11}$, the coefficients of $u_{11}^{2}$ can be further estimated by:
\eq{\label{ineq:C_11-non-positive}
&\mcC_{11}+\frac{1}{2}\abs{\cos\theta}^{2}b_{1}^{2}\frac{W^{2}(1+\log v)}{v\log v}\\
\overset{\eqref{eq:coefficients of uii^2}}{=}&\frac{2W(W+\abs{\cos\theta}u_nb_n)-\abs{\cos\theta}^{2}W^{2}b_1^{2}}{2v}-O\left(\frac{W}{\log v}\right)\\
\ge\,& W^{2}\frac{1-\abs{\cos\theta}^{2}+(\abs{\cos\theta}b_{n}+1)^{2}}{2v}-O\left(\frac{W}{\log v}\right).
}
For the coefficient of $u_{n1}^{2}$,  recall the definition of $\mcC_{n1}$ in \eqref{eq:J},  we compute
\eq{
&\mcC_{n1}-A^{2}\frac{W^{2}(1+\log v)}{v\log v}\\
{=}&\frac{2u_{n}^{2}}{W}+2\abs{\cos\theta}u_{n}b_{n}+\frac{2}{W}-\frac{(1+\log v)}{v\log v}\left(u_{n}+\abs{\cos\theta}Wb_n\right)^{2}-C\frac{W}{\log v}-\frac{A^{2}W^{2}(1+\log v)}{v\log v}\\
\overset{\eqref{eq:AW-v}}{=}&\frac{2u_{n}^{2}}{W}+2\abs{\cos\theta}u_{n}b_{n}+\frac{2}{W}-\frac{2(1+\log v)}{v\log v}\left(u_{n}+\abs{\cos\theta}Wb_n\right)^{2}-O\left(\frac{W}{\log v}\right)\\
\overset{\eqref{eq:v-z_0}}{=}&2\frac{\left(u_{n}^{2}+\abs{\cos\theta}u_nb_nW\right)(W+\abs{\cos\theta}u_nb_n)
-W(u_{n}+\abs{\cos\theta}Wb_n)^{2}}{Wv}-O\left(\frac{W}{\log v}\right)\\
=&\frac{-2\abs{\cos\theta}u_nb_n-2\abs{\cos\theta}^2Wb_n^2}{Wv}-O\left(\frac{W}{\log v}\right),
}
where we have used $u_n^2-W^2=-1$ in the last equality.
Recall \eqref{eq:u_ni}, by Cauchy inequality we thus find
\eq{\label{ineq:C_n1-u_n1^2-non-positive}
&\left(\mcC_{n1}-A^{2}\frac{W^{2}(1+\log v)}{v\log v}\right)\frac{u_{n1}^{2}}{v\log v}\\
=&\left(\frac{-2\abs{\cos\theta}u_nb_n-2\abs{\cos\theta}^2Wb_n^2}{Wv^{2}\log v}-C\frac{W}{v(\log v)^{2}}\right)\left(-\frac{\abs{\cos\theta}}{A}b_{1}u_{11}-\frac{\psi^\ast_{1}}{A\psi^\ast}v\log v\right)^{2}\\
\ge&-C\frac{1}{(\log v)^{2}}\left(\frac{\abs{\cos\theta}^{2}}{A^{2}}b_{1}^{2}u_{11}^{2}+\frac{1}{A^{2}}\left(\frac{\psi^\ast_{1}}{\psi^\ast}\right)^{2}v^{2}(\log v)^{2}\right).
}
Going back to \eqref{ineq:P_1-1}, we can now use \eqref{ineq:C_11-non-positive} and \eqref{ineq:C_n1-u_n1^2-non-positive} to deduce
\eq{\label{ineq:P_1}
\msP_{1}
\ge&\frac{1}{v\log v}\left(\mcC_{nn}u_{nn}^{2}+\sum_{i=2}^{n-1}\mcC_{ni}u_{ni}^{2}+\sum_{i=2}^{n-1}\mcC_{ii}u_{ii}^{2}\right)\\
&+\frac{1}{v\log v}\left(W^{2}\frac{\sin^{2}\theta+(\abs{\cos\theta}b_n+1)^{2}}{2v}-C\frac{W}{\log v}\right)u_{11}^{2}\\
&-\frac{C}{(\log v)^{2}}\left(\frac{\abs{\cos\theta}^{2}}{A^{2}}b_{1}^{2}u_{11}^{2}+\frac{1}{A^{2}}\left(\frac{\psi^\ast_{1}}{\psi^\ast}\right)^{2}v^{2}(\log v)^{2}\right)\\
\ge&\frac{1}{v\log v}\left(\mcC_{nn}u_{nn}^{2}+\sum_{i=2}^{n-1}\mcC_{ni}u_{ni}^{2}+\sum_{i=2}^{n-1}\mcC_{ii}u_{ii}^{2}\right)
+\frac{1}{2v^2\log v}W^{2}{\sin^{2}\theta}u_{11}^{2}
-\frac{C}{A^{2}}\left(\frac{\psi^\ast_{1}}{\psi^\ast}\right)^{2}v^{2},
}
where the last inequality holds because for sufficiently large $\abs{Du(z_0)}$, there holds
\eq{
\left(\frac{W^2(\abs{\cos\theta}b_n+1)^2}{2v^2\log v}
-C\frac{W}{v(\log v)^2}-C\frac{\abs{\cos\theta}^2b_1^2}{A^2(\log v)^2}\right)u_{11}^2
\geq0.
}

\noindent{\bf Step 4.
We complete the proof by using the interior maximality \eqref{ineq:maximun-principle}.
}

Let us first collect up the estimates resulting from {\bf Step 3}.
Using \eqref{ineq:P_1}, \eqref{eq:lower bound og unn2-1}  and  the fact that $\mcC_{ni}\geq0$ (recall \eqref{eq:coefficients of uni^2}), we get
\eq{
\msP_0
=&\msP_{1}-W^{2}(1+\log v)\left(\frac{\psi^\ast_{1}}{\psi^\ast}\right)^{2}\\
\ge&\frac{1}{v\log v}\left(\mcC_{nn}u_{nn}^{2}+\sum_{i=2}^{n-1}\mcC_{ni}u_{ni}^{2}+\sum_{i=2}^{n-1}\mcC_{ii}u_{ii}^{2}\right)
+\frac{\sin^2\theta}{2v^2\log v}W^{2}u_{11}^{2}-2W^{2}(1+\log v)\left(\frac{\psi^\ast_{1}}{\psi^\ast}\right)^{2}\\
\overset{\eqref{eq:lower bound og unn2-1}}{\ge}&\frac{\mcC_{nn}}{v\log v}\left(\frac{3}{4A^{2}}u_{n}^{2}(v\log v)^{2}\left(\frac{\varphi'}{\varphi}+\frac{1}{2N_\ast r(\psi^\ast)^{\frac12}}\right)^{2}-C(\theta)\sum_{i=1}^{n-1}u_{ii}^{2}\right)\\
&{+\frac{\sin^2\theta}{2v^2\log v}W^{2}u_{11}^{2}+\frac{1}{v\log v}\sum_{i=2}^{n-1}\mcC_{ii}u_{ii}^{2}}
-2W^{2}(1+\log v)\left(\frac{\psi^\ast_{1}}{\psi^\ast}\right)^{2}.
}
Because $\mcC_{nn}$ resulting from \eqref{eq:coefficients of unn^2} is of order $O\left(\frac1{W}\right)$,
collecting all the terms involving $u_{11}^2$ in the last inequality, we easily see that for sufficiently large $\abs{Du(z_0)}$,
\eq{
\left(\frac{\sin^2\theta}{2v^2\log v}W^{2}-C(\theta)\frac{\mcC_{nn}}{v\log v}\right)u_{11}^2\geq0.
}
On the other hand, by virtue of the {\bf Claim},
$\mcC_{ii}$ are positive and of order $O(\frac{W^2}v)$ for $i\in\{2,\ldots,n-1\}$, it follows that
\eq{
\sum_{i=2}^{n-1}\left(\frac{\mcC_{ii}}{v\log v}-C(\theta)\frac{\mcC_{nn}}{v\log v}\right)u_{ii}^2
\geq0.
}
For sufficiently large $|Du|$, we thus find
\eq{\label{positive lower bound of P0}
\msP_0
\ge&\frac{\mcC_{nn}}{v\log v}\frac{3}{4A^{2}}u_{n}^{2}(v\log v)^{2}\left(\frac{\varphi'}{\varphi}+\frac{1}{2N_\ast r(\psi^\ast)^{\frac12}}\right)^{2}-2W^{2}(1+\log v)\left(\frac{\psi^\ast_{1}}{\psi^\ast}\right)^{2}\\
\overset{\eqref{eq:psi_i^ast}}{\geq}&\frac14u_n^2\log v\left(\frac{\varphi'}{\varphi}+\frac{1}{2N_\ast r(\psi^\ast)^{\frac12}}\right)^{2}-8W^{2}(1+\log v)\frac{\abs{D\mfQ}^2}{\psi^\ast}\\
\overset{\eqref{ineq:DQ}}{\geq}&\frac14u_n^2\log v\left(\left(\frac{\varphi'}{\varphi}\right)^{2}+\frac{1}{4N_\ast^{2}r^{2}\psi^\ast}+\frac{\varphi'}{\varphi}\frac{1}{N_\ast r(\psi^\ast)^{\frac12}}\right)-9W^{2}\log v\frac{\left( 2+\frac{C_+}{2N_\ast}\right)^2}{r^{2}\psi^\ast}\\
\geq&\frac14u_n^2\log v\left(\frac{\varphi'}{\varphi}\right)^{2}+\frac14u_n^2\log v\frac{\varphi'}{\varphi}\frac{1}{N_\ast r(\psi^\ast)^{\frac12}},
}
where to derive the last inequality, we have used the trivial fact that $C_+=\frac1{36}\frac {\abs{\cos\theta}(1-\sin\theta)}{(1+\frac{|\cos\theta |}{\sin\theta})}\leq\frac2{36}=2N_\ast$, so that the last term on the third inequality $\geq-W^2\log v\frac{9^2}{r^{2}\psi^\ast}$, canceling with the term $\frac14u^2_n\log v\frac1{4N^2_\ast r^2\psi^\ast}$.

We are now ready to finish the proof.
Back to \eqref{ineq:maximun-principle}, we can now estimate
\eq{
0
\geq&a^{ij}\left(\frac{\varphi_{ij}}{\varphi}-\frac{\varphi_{i}}{\varphi}\frac{\varphi_{j}}{\varphi}+\frac{\psi^\ast_{ij}}{\psi^\ast}-\frac{\psi^\ast_{i}}{\psi^\ast}\frac{\psi^\ast_{j}}{\psi^\ast}\right)+\msP\\
\overset{\eqref{eq:a^ij-varphi^ij}}{\geq}&{
-u_n^2\left(\frac{\varphi'}{\varphi}\right)^2}+a^{ij}\left(\frac{\psi_{ij}^\ast}{\psi^\ast}-\frac{\psi^\ast_{i}}{\psi^\ast}\frac{\psi^\ast_{j}}{\psi^\ast}\right)
+\msP,
}
where by direct computation
\eq{
\psi_i^\ast
=&2(\psi^\ast)^\frac12\left(\mfQ_i+\frac{u_i}{2N_\ast r}\right),\\
\psi^\ast_{ij}
=&(\psi^\ast)^{-\frac12}\psi^\ast_j\left(\mfQ_i+\frac{u_i}{2N_\ast r}\right)+2(\psi^\ast)^\frac12\left(\mfQ_{ij}+\frac{u_{ij}}{2N_\ast r}\right)
=\frac{\psi^\ast_i\psi^\ast_j}{2\psi^\ast}+2(\psi^\ast)^\frac12\mfQ_{ij}+2(\psi^\ast)^\frac12\frac{u_{ij}}{2N_\ast r}.
}
Recalling \eqref{eq:a^ij-z_0}, we have at $z_0$,
\eq{
a^{ij}\frac{\psi_i^\ast\psi_j^\ast}{(\psi^\ast)^2}
=W^2\sum_{i=1}^{n-1}\left(\frac{\psi^\ast_i}{\psi^\ast}\right)^2+\left(\frac{\psi^\ast_n}{\psi^\ast}\right)^2
=4W^2\frac{\mfQ_i^2}{\psi^\ast}
+4\frac{\mfQ_n^2}{\psi^\ast}+4\frac{\mfQ_nu_n}{N_\ast\psi^\ast r}+u_n^2\frac1{N_\ast^2\psi^\ast r^2},
}
and hence by \eqref{eq:MSE} and \eqref{ineq:DQ}, together with the fact that $-\frac1{(\psi^\ast(z_0))^\frac12}\geq -\frac1{\psi^\ast(z_0)}$ for $
0< \psi^*\le 1$, we find
\eq{
a^{ij}\left(\frac{\psi^\ast_{ij}}{\psi^\ast}-\frac{\psi_i^\ast\psi_j^\ast}{(\psi^\ast)^2}\right)
\geq&-CW^2\frac1{\psi^\ast r^2}
\geq&-CW^2\frac1{\psi^\ast r^2}.
}
Combing with \eqref{ineq:P-1} and \eqref{positive lower bound of P0}, this  shows that
\eq{\label{ineq:psi^ast-logv-conclusion}
0
\geq&
-u^2_n\left(\frac{\varphi'}{\varphi}\right)^2
-CW^2\frac1{\psi^\ast r^2}+\msP\\
{\geq}&\frac14u_n^2\log v\left(\frac{\varphi'}{\varphi}\right)^{2}+\underbrace{\frac14u_n^2\log v\frac{\varphi'}{\varphi}\frac{1}{N_\ast r(\psi^\ast)^{\frac12}}}_{\geq0}
-u^2_n\left(\frac{\varphi'}{\varphi}\right)^2
-CW^2\frac1{\psi^\ast r^2}\\
\geq&\frac15u_n^2\log v\left(\frac{\varphi'}{\varphi}\right)^{2}-CW^2\frac1{\psi^\ast r^2},
}
which recovers an estimate of the form \eqref{ineq:final-esti-step3}, so that a similar argument as the end of {\bf Step 2} in the proof of Theorem \ref{Thm:gradient-estimate} will lead to the following estimate:
\eq{
\sup_{\overline{W_{r_\ast}}}\abs{Du}
\leq \frac{1}{1-|\cos\theta|}\exp\left(C_1+C_2\frac{M}{r}+C_3\frac{M^2}{r^2}\right),
}
where $C_1,C_2,C_3$ are positive constants depending only on $n,\theta$.
Combining this estimate with \eqref{esti:Thm3.3-bdry-maximum},
then letting $r\ra\infty$ (recall that $\lim_{r\ra\infty}W_{r_\ast}=\mathbb{R}^n_+$), we finally deduce \eqref{conclu:gradient-estimate-bounded-above-L}.
The proof is thus completed.
\end{proof}

The case that $u$ is bounded from below by some linear function on $\mathbb{R}^n_+$ follows as a corollary of the above theorem.
\begin{corollary}\label{Thm:gradient-esti-u>=0}
Let
$\theta\in(0,\pi)$,
let $u$ be a $C^2$-solution of the mean curvature equation \eqref{eq:MSE}, such that its graph $\S$ is a  capillary minimal graph in the sense of Definition \ref{Defn:capillary-graph}.
Assume that 
$u$ has linear growth on $\overline{\mathbb{R}^n_+}$, namely, $\abs{u(x)}\leq C_{0}(1+\abs{x})$ for some constant $C_0>0$.

There exists a positive constant $\widehat\Lambda=\widehat\Lambda(n,\theta,C_0)$ with the following property:
If $u$ is bounded from below by a linear function $L$ on $\mathbb{R}^n_+$, with $\abs{DL}\leq\frac1{36}\frac {\abs{\cos\theta}(1-\sin\theta)}{(1+\frac{|\cos\theta |}{\sin\theta})}\eqqcolon C_\theta$,
then
\eq{
\sup_{\overline{\mathbb{R}^n_+}}\abs{Du}
\leq\widehat\Lambda.
}
\end{corollary}

\begin{proof}
Consider the function $-u$, which is by assumption bounded from above by a linear function $-L$ on $\mathbb{R}^n_+$, with $\abs{D(-L)}\leq C_\theta$.
Moreover, the graph of $-u$ is a capillary minimal graph in the sense of Definition \ref{Defn:capillary-graph} with capillary angle $\pi-\theta$.

Applying Theorem \ref{Thm:gradient-esti-u<=0} we then obtain the required estimate.
\end{proof}

\section{Global gradient estimates for one-sided bounded solution}\label{Sec-4}
Our goal of this section is to obtain Theorem \ref{Thm:intro-gradient-sign},
we start with the following key lemma.
\begin{lemma}\label{Lem:gradient-estimate-bdry-u>0}
Let
$\theta\in(0,\pi)$,
let $u$ be a $C^2$-solution of the mean curvature equation \eqref{eq:MSE}, such that its graph $\S$ is a  capillary minimal graph in the sense of Definition \ref{Defn:capillary-graph}.
There exists a positive constant $\tilde\Lambda=\tilde\Lambda(n,\theta)$ with the following property:
If $u$ is a negative function on $\mathbb{R}^n_+$, then
\eq{
\sup_{\p\mathbb{R}^n_+}\abs{Du}
\leq\tilde\Lambda.
}
\end{lemma}
\begin{proof}
Recalling Remark \ref{Rem:theta>pi/2},
in the following we only consider those $\theta\in(0,\pi)\setminus\{\frac\pi2\}$.

We modify the proof of Theorem \ref{Thm:gradient-esti-u<=0} and prove gradient estimate at any fixed $p=(0,p')\in\p\mathbb{R}^n$.
For any $r>0$ sufficiently large, we consider the function
\eq{
G^\ast(x)
=\varphi(u(x))\psi^\ast(x)\log v(x),
}
where $\psi^\ast(x)=\left(1-\frac{(x_{1}-\abs{\cos\theta}r)^{2}+\sin^{2}\theta\abs{x'-p'}^{2}}{r^{2}}+\frac{u(x)}{2N_\ast r}\right)^{2}$; $\varphi(s)=\frac{s}{2M_\ast}+1$, and $M_\ast=u(p)+r$ (which is positive whenever $r>-2u(p)$).

Without loss of generality, we  assume that $p=0$, so that $M_\ast=u(0)+r$, and $\psi^\ast$ agrees with \eqref{defn:psi^ast} ($L$ therein chosen as $0$). We can now follow the proof of Theorem \ref{Thm:gradient-esti-u<=0} almost line by line. 

As in the proof of Theorem \ref{Thm:gradient-esti-u<=0},
define $W_r$ by \eqref{defn:W_r} (in this case we do not have to define the set $W_{r_\ast}$), since $u<0$ we have $W_r\subset E_r$.
Obviously, $0\in W_r$ for large $r$, and we have 
\eq{
G^\ast(0)
=\left(1+\frac{u(0)}{2(u(0)+r)}\right)\left(\sin^2\theta+\frac{u(0)}{2N_\ast r}\right)^2\log v(0)>\frac{\sin^2\theta}8\log v(0)
}
for sufficiently large $r>\max\{-2u(0), \frac{-u(0)}{N_\ast\sin^2\theta}\}$.

Our goal is to prove the following estimate
\eq{\label{esti:Lem4.1-interior+bdry}
\left(1+\frac{u(0)}{2(u(0)+r)}\right)\left(\sin^2\theta+\frac{u(0)}{2N_\ast r}\right)^2\log v(0)
\leq C(n,\theta)\left(\frac{u(0)+r}r\right)^2+C(n,\theta).
}
To this aim, we assume
\eq{
\max_{\overline{W_r}}G^\ast
=G^\ast(z_0)>0.
}
Clearly $z_0\in\overline{W_r}\setminus\p_{rel}W_r$,
and we consider the following two cases:

{\bf Case 1. }$z_0\in\p W_r\setminus\p_{rel}W_r$.

Following the computations in Theorem \ref{Thm:gradient-esti-u<=0} {\bf Step 2}, we find at $z_0$ (recalling \eqref{ineq:max-bdry-u<=0-hopf-lemma})
\eq{
0
\geq\frac1{r(\psi^\ast)^\frac12}\left(2\abs{\cos\theta}(1-\abs{\sin\theta})\sqrt{1+\abs{\bar Du}^2}\abs{\bar Du}
\right)-\frac1{2M_\ast\varphi}(\abs{\cos\theta}W).
}
Recalling \eqref{eq:bar-Du-u_1-bdry}, this implies that $\varphi(u(z_0))\abs{Du(z_0)}\leq C(\theta)\frac{r(\psi^\ast)^\frac12}{M_\ast}\leq C(\theta)\frac{r}{r+u(0)}\leq C(\theta)$ by our choice of $M_\ast$ and the fact that $0\leq\psi^\ast\leq1$ on $W_r$.
Therefore,
since $z_0$ is the maximum point, we have
\eq{
\left(1+\frac{u(0)}{2(u(0)+r)}\right)\left(\sin^2\theta+\frac{u(0)}{2N_\ast r}\right)^2\log v(0)
=G^\ast(0)
\leq G^\ast(z_0)
=\varphi(u(z_0))\psi^\ast(z_0)\log v(z_0)
\leq C(\theta).
}

{\bf Case 2. }$z_0\in W_r$.

Following the computations in Theorem \ref{Thm:gradient-esti-u<=0} {\bf Steps 3,4}, we arrive at $z_0$ (recalling \eqref{ineq:psi^ast-logv-conclusion})
\eq{
0\geq
\frac15u_n^2\log v\left(\frac{\varphi'}{\varphi}\right)^{2}
-C(n,\theta)W^2\frac1{\psi^\ast r^2}.
}
Therefore
\eq{
\psi^\ast(z_0)\log v(z_0)
\leq C(n,\theta)\left(\varphi\frac{M_\ast}r\right)^2
\leq C(n,\theta)\left(\frac{u(0)+r}r\right)^2.
}
This in turn implies
\eq{
\left(1+\frac{u(0)}{2(u(0)+r)}\right)\left(\sin^2\theta+\frac{u(0)}{2N_\ast r}\right)^2\log v(0)
=&G(0)
\leq G(z_0)\\
=&\left(1+\frac{u(z_0)}{2(u(0)+r)}\right)\psi^\ast(z_0)\log v(z_0)
\leq C(n,\theta)\left(\frac{u(0)+r}r\right)^2,
}
where we have used $u<0$ to derive the last inequality.

Combining {\bf Cases 1}, {\bf 2}, we obtain \eqref{esti:Lem4.1-interior+bdry}.
Letting $r\ra\infty$, we deduce as required that $\abs{Du(0)}\leq C(n,\theta)$.
This completes the proof.

\end{proof}

\begin{theorem}\label{Thm:u>0}
Let
$\theta\in(0,\pi)$,
let $u$ be a $C^2$-solution of the mean curvature equation \eqref{eq:MSE}, such that its graph $\S$ is a  capillary minimal graph in the sense of Definition \ref{Defn:capillary-graph}.
There exists a positive constant $\widetilde\Lambda=\widetilde\Lambda(n,\theta)$ with the following property:
If $u$ is a negative function on $\mathbb{R}^n_+$, then for any $p\in{\mathbb{R}^n_+}$, there holds
\eq{
\sup_{\overline{\mathbb{R}^n_+}}\abs{Du}
\leq\widetilde\Lambda.
}
\end{theorem}

\begin{proof}[Proof of Theorem \ref{Thm:u>0}]
When we have the boundary estimate, Lemma \ref{Lem:gradient-estimate-bdry-u>0}, the Theorem follows from \cite[Theorem 1.4]{CMMR22}, which is a more general result. Since the proof therein is completely different from our context, for the convenience of the reader we 
provide a proof, which is a modification of the one of Theorem \ref{Thm:gradient-estimate}.
Fix an arbitrary $p\in\mathbb{R}^n_+$,
for any $r>-2u(p)>0$ sufficiently large,
we consider the function
\eq{
G(x)
=\varphi(u(x))\psi(x)\log W(x),
}
where $\varphi(s)=\frac{s}{2M}+1$ with $M=u(p)+r$; $\psi(x)$ is defined as in Lemma \ref{Lem:cut-off} but with $\theta$ therein chosen as $\frac\pi2$ and center chosen as $p$, namely, $\psi(x)=\left(1-\frac{\abs{x-p}^2}{r^2}\right)^2$; also recall that $W(x)=\sqrt{1+\abs{Du(x)}^2}$.

Put $D_r\coloneqq B_r(p)\cap\mathbb{R}^n_+$.
Note that $G(p)=\left(1+\frac{u(p)}{2(u(p)+r)}\right)\log W(p)$,
hence we assume
\eq{
\max_{\overline{D_r}}G
=G(z_0)>0.
}
Clearly, $z_0\in\overline{D_r}\setminus\p_{rel}D_r$, and we consider the following two cases:

{\bf Case 1. }$z_0\in\p\mathbb{R}^n_+$.

In this case, by virtue of Lemma \ref{Lem:gradient-estimate-bdry-u>0} we have
\eq{
\left(1+\frac{u(p)}{2(u(p)+r)}\right)\log W(p)
=G(p)\leq G(z_0)
=&\left(1+\frac{u(z_0)}{2(u(p)+r)}\right)\psi(z_0)\log W(z_0)
\leq C(n,\theta).
}

{\bf Case 2. }$z_0\in D_r$.

Following the computations in Theorem \ref{Thm:gradient-estimate} {\bf Step 2} (with $\theta$ therein chosen as $\frac\pi2$ thanks to our choice of $\psi$. In particular, the crucial estimate \eqref{ineq:C_i} holds), we arrive at (recalling \eqref{ineq:final-esti-step2}, \eqref{ineq:C_i})
\eq{
\varphi(z_1)\psi(z_1)\log W(z_1)
\leq C(n)\left(\frac{M}r\right)^2+C(n,\theta)
=C(n)\left(\frac{u(p)+r}r\right)^2+C(n,\theta).
}
Therefore, since $z_1$ is the maximum point, we have
\eq{
\left(1+\frac{u(p)}{2(u(p)+r)}\right)\log W(p)
=G(p)
\leq&G(z_1)
=\left(1+\frac{u(z_1)}{2(u(p)+r)}\right)\psi(z_1)\log W(z_1)\\
\leq& C(n)\left(\frac{u(p)+r}r\right)^2+C(n,\theta).
}
Combining {\bf Cases 1,2}, we obtain 
\eq{
\left(1+\frac{u(p)}{2(u(p)+r)}\right)\log W(p)\le C(n)\left(\frac{u(p)+r}r\right)^2+C(n,\theta).
}
Letting $r\ra\infty$, we obtain as required that $\abs{Du(p)}\leq C(n,\theta)$, and we complete the proof.
\end{proof}

\begin{proof}[Proof of Theorem \ref{Thm:intro-gradient-sign}]
To prove the theorem, note that up to plus or minus a constant from $u$, it suffices to prove gradient estimates for those $u$ which are either negative or positive on $\mathbb{R}^n_+$.

For the former case, we directly apply Theorem \ref{Thm:u>0}; for the later case, we apply Theorem \ref{Thm:u>0} to the function $-u$ (which is a negative function and its graph is a capillary minimal graph in the sense of Definition \ref{Defn:capillary-graph} with capillary angle $\pi-\theta$).
This finishes the proof.

\end{proof}

\section{Liouville-type theorems}\label{Sec-5}

\begin{proof}[Proof of Theorem \ref{Thm:Liouville}]

We first prove ($i$) and ($ii$).
Denote $u_R(x)=u(Rx)/R$, then $D u_R(x)=D u(Rx)$ with $\abs{u_R}(x)\le C_0$  and by Theorem \ref{Thm:gradient-estimate} $\abs{Du_R(x)}\le\Lambda$ in $B_1^+=B_1(0)\cap\{x_1>0\}$.
Moreover, $u_R(x)$ satisfies
\eq{
\div\left(\frac{D u_R}{\sqrt{1+\abs{D u_R}^2}}\right)
=0\quad\text{in}\quad B_1^+,
}
with $(u_R)_1=\cos\theta \sqrt{1+\abs{D u_R}^2}$ on $\p B_1^+\cap\p\mathbb{R}^n_+$.
Therefore, by standard estimate (see e.g., \cite[Section 10.2]{LU68}) we have $\abs{u_R}_{C^{1,\alpha}(B_{\frac12}^+)}\le C$, where $C$ is a positive constant independent of $R$.
In particular, this yields $\Abs{D u_R(x)-D u_R(0)}\le C\abs{x}^{\alpha}$ for any $x\in B_{\frac12}^+$ and thus, for any $y\in B_{R/2}^+$,
\eq{
\Abs{D u(y)-D u(0)}
\le C\frac{\abs{y}^{\alpha}}{R^{\alpha}}.
}
For any fixed $y$, letting $R\rightarrow \infty$ we obtain $\Abs{D u(y)-D u(0)}=0$, and thus $u$ is affine.

To prove ($iii$), note that by Theorem \ref{Thm:gradient-esti-u<=0} and Corollary \ref{Thm:gradient-esti-u>=0}, we again have $\abs{Du_R(x)}\leq\Lambda$ in $B_1^+$, so we conclude as above that $u$ is affine, which completes the proof.
\end{proof}
The proof of Theorem \ref{Thm:Liouville-positive} is essentially the same, thanks to Theorem \ref{Thm:intro-gradient-sign}.


\appendix
\section{A calibration argument for capillary minimal graphs 
}\label{App-1}

Let ${\rm vol}=\rd x_1\wedge\ldots\wedge \rd x_{n+1}$ be the canonical volume form of $\mathbb{R}^{n+1}$.
Given $\theta\in(0,\pi)$,
let $u$ be a smooth function on $\mathbb{R}^n_+$ and $\S$ be its corresponding graph, such that $\S$ is a capillary minimal graph in the sense of Definition \ref{Defn:capillary-graph}.

\begin{definition}[Capillary calibration]
\normalfont
Let $\nu_\theta\coloneqq\nu-\cos\theta e_1$ be a vector field defined on $\overline{\mathbb{R}^n_+}$, where $\nu$ is the upwards-pointing unit normal of $\S\subset\mathbb{R}^{n+1}$, defined by \eqref{eq:nu}.
We call $\nu_\theta$ the \textit{capillary normal} of $\S$ (with respect to the capillary angle $\theta$).
Extending $\nu, \nu_\theta$ to be defined on $\overline{\mathbb{R}^{n+1}_+}$ by simply letting
$\nu(x,x_{n+1})=\nu(x)$ and
$\nu_\theta(x,x_{n+1})=\nu_\theta(x)$ for any $x\in\overline{\mathbb{R}^n_+}$.
The $n$-form $\om_\theta$, defined by
\eq{
\om_\theta
=\iota_{\nu_\theta}{\rm vol},
}
is called \textit{capillary calibration}.
\end{definition}

\begin{lemma}\label{Lem:capillary-calibration}
The capillary calibration
$\om_\theta$ satisfies
\begin{enumerate}
    \item [($i$)] $\rd\om_\theta=0$;
    \item [($ii$)] $\om_\theta\mid_{\p\mathbb{R}^{n+1}_+}=0$;
    \item [($iii$)]
    For any positively oriented orthonormal basis $\{\tilde\tau_1,\ldots,\tilde\tau_n\}$ of a hyperplane $\mbP$ in $\mathbb{R}^{n+1}$
    (i.e., $\{\nu_\mbP,\tilde\tau_1,\ldots,\tilde\tau_n\}$ agrees with the orientation ${\rm vol}$ of $\mathbb{R}^{n+1}$), there holds
    \eq{
    \om_\theta\mid_{(x,x_{n+1})}(\tilde\tau_1,\tilde\tau_2,\ldots,\tilde\tau_n)\leq1-\cos\theta\left<\nu_\mbP,e_{n+1}\right>,
    }
    Moreover, equality holds if and only if $\mbP$ is a tangent space of $T_{(x,u(x))}\S$.
\end{enumerate}
\end{lemma}
\begin{proof}
Note that since $\nu_\theta=\nu-\cos\theta e_1$, we can write
\eq{
\om_\theta
=\om-\cos\theta\iota_{e_1}{\rm vol},
}
where $\om=\iota_\nu{\rm vol}$ is the classical calibration, and satisfies $\rd\om=0$ since $\S$ is a minimal graph.
On the other hand, it is easy to see that $\rd(\iota_{e_1}{\rm vol})=0$, which proves ($i$).

Conclusion ($ii$) simply follows from the fact that $\nu_\theta(x)\in\p\mathbb{R}^{n+1}_+$ for any $x\in\p\mathbb{R}^n_+$, since
\eq{
\nu_\theta(x)
=\frac{(-Du(x),1)}{\sqrt{1+\abs{Du(x)}^2}}-\cos\theta e_1
\overset{\eqref{condi:capillary-bdry-2}}{=}\frac{(0,-\bar Du(x),1)}{\sqrt{1+\abs{Du(x)}^2}},\quad\forall x\in\p\mathbb{R}^n_+.
}

Conclusion ($iii$) follows from the following two facts:
For any positively oriented orthonormal basis $\{\tilde\tau_1,\ldots,\tilde\tau_n\}$ of a hyperplane $\mbP$ in $\mathbb{R}^{n+1}$,
\begin{enumerate}
    \item The classical calibration $\om$ satisfies:
    \eq{
    \om\mid_{(x,x_{n+1})}(\tilde\tau_1,\ldots,\tilde\tau_n)
    \leq1,
    }
    and equality holds if and only if $\mbP=T_{(x,u(x))}\S$.
    \item $-\cos\theta\left(\iota_{e_1}{\rm vol}\right)(\tilde\tau_1,\ldots,\tilde\tau_n)=\left<-\cos\theta e_1,\nu_\mbP\right>$.
\end{enumerate}
This completes the proof.
\end{proof}

\begin{proposition}\label{Prop:capillary-calibration-minimizer}
Let $\theta\in(0,\pi)$, $u$ be a smooth function on $\mathbb{R}^n_+$ and $\S$ be its corresponding graph, such that $\S$ is a capillary minimal graph in the sense of Definition \ref{Defn:capillary-graph}.
Then $\S$ is a minimizer of the capillary area functional \eqref{defn:capillary-area-functional} in the following sense:

Let $E\subset\mathbb{R}^n_+$, denote the truncated hypersurface $\S\cap(E\times\mathbb{R})$ by $\tilde\S$, and $\S'\subset E\times\mathbb{R}$ is any other hypersurface with
\eq{\label{condi:competitor}
\p\S'\cap\mathbb{R}^{n+1}_+
=\p\tilde\S\cap\mathbb{R}^{n+1}_+,
}
serving as a competitor.
Then
\eq{
\int_{\S'}F_\theta(\nu(p))\rd\mcH^n(p)
\geq\int_{\tilde\S}
F_\theta(\nu(p))\rd\mcH^n(p).
}
\end{proposition}
\begin{proof}

By \eqref{condi:competitor} there exists a domain $\mathscr{D}\subset\mathbb{R}^{n+1}_+$ such that $\p\mathscr{D}=\tilde\S\cup\S'\cup\left(\p\mathscr{D}\cap\p\mathbb{R}^{n+1}_+\right)$, where the set $\p\mathscr{D}\cap\p\mathbb{R}^{n+1}_+$ is known as the wetting region associated with $\mathscr{D}$.
Stokes' theorem, in conjunction with Lemma \ref{Lem:capillary-calibration} ($i$) ($ii$), gives
\eq{
0=\int_{\mathscr{D}}\rd\om_\theta
=\int_{\tilde\S}\om_\theta-\int_{\S'}\om_\theta.
}
Combining with \ref{Lem:capillary-calibration} ($iii$) we obtain as required that
\eq{
\int_{\tilde\S}F_\theta(\nu(p))\rd\mcH^n(p)
=\int_{\tilde\S}\om_\theta
=\int_{\S'}\om_\theta
\leq\int_{\S'}F_\theta(\nu(p))\rd\mcH^n(p),
}
which completes the proof.
\end{proof}

\begin{remark}
\normalfont
In view of \eqref{eq:F_theta-energy-capillary-energy}, we deduce by virtue of Proposition \ref{Prop:capillary-calibration-minimizer} that any capillary minimal graph over $\mathbb{R}^n_+$, is stable for the capillary functional.
Moreover, a standard argument in conjunction with Proposition \ref{Prop:capillary-calibration-minimizer} shows that any such $\S$ automatically has Euclidean area growth.
With these two ingredients, the Bernstein theorems for capillary hypersurface in the half-space \cite{HS23,MP21,LZZ24} apply and yield a half-space Bernstein theorem for minimal graphs with capillary boundary.
\end{remark}

\section{A Bernstein-type theorem for minimal graphs with free boundary}\label{App-2}

In this section we consider the following minimal surface equation:
Let $\Om\subset\mathbb{R}^n$ be an unbounded domain (open, connected) with $C^2$-boundary $\p\Om$,
and $\bar N$ be the outer unit normal of $\Om$ along $\p\Om$.
Let $u$ be a smooth function defined on $\Om$ and denote its graph by $\S=\left\{(x,u(x)):x\in\Om\right\}$, which is a hypersurface in the cylinder $\Om\times\mathbb{R}$.
We say that $u$ solves the \textit{free boundary minimal surface equation on $\Om$}, if $u$ satisfies
\eq{\label{defn:MSE-free-bdry}
\begin{cases}
    {\rm div}\left(\frac{Du}{\sqrt{1+\abs{Du}^2}}\right)=0,\quad&\text{on }\Om,\\
    \left<Du(x),\bar N(x)\right>=0,
    \quad&\forall x\in\p\Om.
\end{cases}
}
Equivalently, the graph $\S$ is a free boundary minimal hypersurface in $\Om\times\mathbb{R}$ (note that, with a slight abuse of notation, if we denote by $\bar N$ the outer unit normal of $\Om\times\mathbb{R}$ along $\p\Om\times\mathbb{R}$, then $\bar N(x,x_{n+1})=\bar N(x)$, for any $x\in\p\Om$, and $x_{n+1}\in\mathbb{R}$).

In this case, the Bernstein-type theorem reads as follows.
\begin{theorem}\label{Thm:Bernstein-free-bdry}
Let $\Om\subset\mathbb{R}^n$ be an unbounded domain with $C^2$-boundary $\p\Om$.
Let $u$ be a smooth solution to the free boundary minimal surface equation on $\Om$.
If $2\leq n\leq 6$, then $u$ is affine.
\end{theorem}

The proof relies on establishing curvature estimate for free boundary minimal graph, which follows essentially from the curvature estimate of immersed/embedded stable free boundary minimal hypersurface by Guang-Li-Zhou \cite{GLZ20}.
Before we proceed to that step, let us first state the following facts on calibration (we continue to use the notations $\nu,\om$ as in Appendix \ref{App-1}).
\begin{lemma}
Under the above notations, the calibration $\om$ satisfies
\begin{enumerate}
    \item [($i$)] $\rd\om=0$;
    \item [($ii$)] $\om\mid_{\p(\Om\times\mathbb{R})}=0$;
    \item [($iii$)] For any positively oriented orthonormal basis $\{\tilde\tau_1,\ldots,\tilde\tau_n\}$ of a hyperplane $\mbP$ in $\mathbb{R}^{n+1}$
    (i.e., $\{\nu_\mbP,\tilde\tau_1,\ldots,\tilde\tau_n\}$ agrees with the orientation ${\rm vol}$ of $\mathbb{R}^{n+1}$), there holds
    \eq{
    \om\mid_{(x,x_{n+1})}(\tilde\tau_1,\ldots,\tilde\tau_n)
    \leq1,
    }
    and equality holds if and only if $\mbP=T_{(x,u(x))}\S$.
\end{enumerate}
\end{lemma}
\begin{proof}
Conclusions ($i$) and ($iii$) are contained in the proof of Lemma \ref{Lem:capillary-calibration}.
To prove ($ii$), it suffices to observe that thanks to the free boundary condition, we have $\nu(x,x_{n+1})\in T_{(x,x_{n+1})}\p(\Om\times\mathbb{R})$ for any $x\in\p\Om$ and $x_{n+1}\in\mathbb{R}$, so that $\om\mid_{\p\Om\times\mathbb{R}}=(\iota_\nu{\rm vol})\mid_{\p\Om\times\mathbb{R}}=0$.
\end{proof}
A standard argument then shows the following area minimizing property.
\begin{proposition}[Area minimizing]\label{Prop:area-minimizing-free-bdry}
Let $\Om\subset\mathbb{R}^n$ be an unbounded domain with $C^2$-boundary $\p\Om$, 
and $u$ be smooth solution to the free boundary minimal surface equation on $\Om$, and let $\S$ be the graph of $u$.
Then $\S$ is a minimizer of the area functional in the following sense:

Let $E\subset\Om$, denote the truncated hypersurface $\S\cap(E\times\mathbb{R})$ by $\tilde\S$, and $\S'\subset E\times\mathbb{R}$ is any other hypersurface with
\eq{
\p\S'\cap\left(\Om\times\mathbb{R}\right)
=\p\tilde\S\cap\left(\Om\times\mathbb{R}\right),
}
serving as a competitor.
Then
\eq{
\mcH^n(\S')
\geq\mcH^n(\tilde\S).
}
In particular, one has
\begin{enumerate}
    \item [($i$)] $\S$ is a stable free boundary minimal hypersurface (graph) in $\Om\times\mathbb{R}$.
    \item [($ii$)] $\S$ satisfies the following Euclidean area growth condition: for any $p\in\mathbb{R}^{n+1}$ and any $r>0$,
    \eq{\label{ineq:Euclidean-growth}
    \mcH^n(\S\cap B_r(p))
    \leq \mcH^n(\mbS^n)r^n
    =(n+1)\om_{n+1}r^n.
    }
\end{enumerate}
\end{proposition}
\begin{remark}[Rescaling]
\normalfont
Let $\lambda_i>0$ and $y_i\in\mathbb{R}^{n+1}$,
let $\eta_i(z)=\lambda_i(z-y_i)$ be a (blow-up) map on $\mathbb{R}^{n+1}$.
Let $\S$ be as in Proposition \ref{Prop:area-minimizing-free-bdry}, then the Euclidean area growth condition also holds for $\eta_i(\S)$.
In fact, we have
\eq{\label{ineq:Euclidean-growth-rescaled}
\mcH^n(\eta_i(\S)\cap B_r(0))
=\lambda^n_i\mcH^n(\S\cap B_{\lambda^{-1}_ir}(y_i))
\overset{\eqref{ineq:Euclidean-growth}}{\leq}(n+1)\om_{n+1}r^n.
}

We also note that the rescaled $\eta_i(\S)$ is a free boundary minimal graph in the rescaled cylinder $\eta_i\left(\Om\times\mathbb{R}\right)$.
\end{remark}

\begin{proposition}[Curvature estimate for free boundary minimal graph]
Let $\Om\subset\mathbb{R}^n$ be an unbounded domain with $C^2$-boundary $\p\Om$, assume WLOG that $0\in\p\Om$.
Let $u$ be smooth and solve the free boundary minimal surface equation on $\Om$, and denote by $\S$ its graph.
If $2\leq n\leq 6$, then for any $R>0$, the curvature estimate holds:
\eq{\label{esti:curvature-free-bdry}
\sup_{p\in\S\cap B_\frac{R}2(0)}\abs{A^\S}(p)
\leq\frac{C_1}R,
}
where $C_1>0$ is a constant depending on $\Om$ and $n$.
\end{proposition}
As said, the proof is essentially given by \cite{GLZ20}.
Here we just sketch it.
\begin{proof}[Sketch of proof]
By the rescaling property it suffices to prove the curvature estimate in $B_\frac12(0)$.
Assume by contradiction the curvature estimate fails, then there exists a sequence $\{\S_i\}_{i\in\mbN}$ of free boundary minimal graphs on $\Om$ such that as $i\ra\infty$,
\eq{
\sup_{p\in\S_i\cap B_\frac12(0)}\abs{A^{\S_i}}(p)\ra\infty.
}
Following \cite[Theorem 4.1, Step 1]{GLZ20}, we obtain a sequence of blow-up maps $\eta_i:\mathbb{R}^{n+1}\ra\mathbb{R}^{n+1}$, given precisely by $\eta_i(z)\coloneqq\lambda_i(z-y_i)$, $z\in\mathbb{R}^{n+1}$, where $\{y_i\}_{i\in\mbN}$ is a sequence of points on $\S_i$ satisfying certain property, and $\lambda_i\coloneqq\abs{A^{\S_i}}(y_i)\ra\infty$.
We then get a blow-up sequence of free boundary minimal graphs $\S'_i\coloneqq\eta_i(\S_i)$ in $\eta_i(\Om\times\mathbb{R})$.
Note that
\begin{itemize}
    \item $\abs{A^{\S'_i}}(0)=\lambda^{-1}_i\abs{A^{\S_i}}(y_i)=1$ for each $i\in\mbN$;
    \item For the blow-up sequence of minimal graphs, the uniform Euclidean area growth condition still holds, thanks to \eqref{ineq:Euclidean-growth-rescaled};
    \item Furthermore, for any fixed $r>0$, the curvatures of $\S'_i$ in the fixed ball $B_r(0)$ are uniformly bounded, provided that $i$ is sufficiently large,
    see \cite[eqn. (4.4)]{GLZ20}.
    This is in fact done by the construction of the sequence of points $\{y_i\in\S_i\}_{i\in\mbN}$.
\end{itemize}
With these properties, we can then use the compactness results for minimal submanifolds (without boundary or with free boundary) with bounded curvature and uniform Euclidean area growth to conclude as in \cite[Step 2]{GLZ20} that, after passing to a subsequence, $\S'_i$ converge smoothly and locally uniformly to
\begin{itemize}
    \item either a complete, embedded stable minimal hypersurface $\S^1_\infty$ in $\mathbb{R}^{n+1}$;
    \item or a embedded, stable free boundary minimal hypersurface $\S^2_\infty$ in the Euclidean half-space $\mathbb{R}^{n+1}_+$, such that $\S^2_\infty$ has non-empty free boundary $\p\S^2_\infty$ with $\p\mathbb{R}^{n+1}_+$.
    Reflecting it across the hyperplane $\p\mathbb{R}^{n+1}_+$ we obtain a complete, embedded stable minimal hypersurface in $\mathbb{R}^{n+1}$.
\end{itemize}
In both cases, the same Euclidean area growth as in \eqref{ineq:Euclidean-growth} is satisfied for all $r>0$, with $\S$ replaced by $\S^1_\infty$ or $\S^2_\infty$.
Also by construction, $\abs{A^{\S^1_\infty}}(0)=1$ or $\abs{A^{\S^2_\infty}}(0)=1$, which contradicts to the classical Bernstein theorem \cite{SSY75,SS81}, that $\S^1_\infty$ or $\S^2_\infty$ has to be flat (see also the recent advance by Bellettini \cite{Bellettini25}, which extends the classical Bernstein theorem for stable minimal immersed hypersurface by Schoen-Simon-Yau \cite{SSY75} to $n=6$; and also the work on $\de$-stable minimal hypersurface by Hong-Li-Wang \cite{HLW24}).
This completes the proof.
\end{proof}

\begin{proof}[Proof of Theorem \ref{Thm:Bernstein-free-bdry}]
Letting $R\ra\infty$ in the curvature estimate \eqref{esti:curvature-free-bdry}, we complete the proof.
\end{proof}


\bibliography{BibTemplate.bib}
\bibliographystyle{amsplain}

\end{document}